\documentclass[10pt]{amsart}

\usepackage{amsmath}
\usepackage[latin1]{inputenc}
\usepackage{amsfonts}
\usepackage{amssymb}

\usepackage{amsthm}
\usepackage{amscd}

\usepackage[latin1]{inputenc}

\usepackage{amsmath}

\usepackage{amsfonts}

\usepackage{amssymb}

\usepackage{amsthm}

\usepackage{graphics}

\newcommand{\mk}{\medskip}

\newcommand{\ZZ}{\mathbb{Z}}

\newcommand{\CC}{\mathbb{C}}

\newcommand{\NN}{\mathbb{N}}

\newcommand{\Glie}{\mathfrak{g}}

\newcommand{\Hlie}{\mathfrak{h}}

\newcommand{\demo}{\noindent {\it \small Proof:}\quad}

\newcommand{\U}{\mathcal{U}}

\newtheorem{thm}{Theorem}[section]

\newtheorem{defi}[thm]{Definition}

\newtheorem{prop}[thm]{Proposition}

\newtheorem{lem}[thm]{Lemma}

\newtheorem{rem}[thm]{Remark}

\title{Geometry of the analytic loop group}

\author{Corrado De Concini}
\address{Dip. Mat. Castelnuovo, Univ. di Roma La Sapienza, Rome, ITALY}
\email{deconcin@mat.uniroma1.it}
\author{David Hernandez}
\address{Univ. Paris Diderot-Paris 7,IMJ - PRG CNRS UMR 7586,
B\^at. Sophie Germain, Case 7012,
75205 Paris Cedex 13, FRANCE}
\email{hernandez@math.jussieu.fr}
\author{Nicolai Reshetikhin}
\address{Dep. of Math., UC Berkeley, 970 Evans Hall, Berkeley, CA 94720, USA}
\email{reshetik@math.berkeley.edu}

\usepackage[all]{xy}
\begin{document}

\begin{abstract} We introduce and study a notion of analytic loop group with a Riemann-Hilbert factorization  relevant for the representation theory of quantum affine algebras at roots of unity $\U_\epsilon(\hat{\Glie})$ with non trivial central charge. We introduce  a Poisson structure and study properties of its Poisson dual group. We prove that the Hopf-Poisson structure is isomorphic to the semi-classical limit of the center of $\U_\epsilon(\hat{\Glie})$ (it is a geometric realization of the center). Then the symplectic leaves, and corresponding equivalence classes of central characters of $\U_\epsilon(\hat{\Glie})$, are parameterized by certain $G$-bundles on an elliptic curve.
\end{abstract}

\maketitle

\tableofcontents

\section{Introduction}

Let $\Glie$ be finite dimensional semi-simple Lie algebra and $\U_q(\Glie)$ the corresponding quantized universal enveloping algebras. When  $q$ is a formal variable the representation theory of $\U_q(\Glie)$ is very similar to the representation theory of $\Glie$. In particular the category of finite dimensional representation is semi-simple and simple (type $1$) finite dimensional  representations are parameterized by dominant integral weights.

When the formal variable $q$ is specialized at a root of unity $\epsilon$ the representation theory changes drastically.
There are many such specializations of $\U_q(\Glie)$. The best
known are with divided powers \cite{lj} and with the big
center \cite{dck}. In the later case the algebra is a
finite  module over its center (see \cite{b2, dc, dck, dcp, dckp, e1, e2} and references therein). The center of such a specialization is a finite module over a commutative Hopf algebra which
is isomorphic to the algebra of polynomial functions on an affine algebraic Lie group whose Lie algebra is $\Glie$. Thus, geometry comes into the picture.

Moreover the representation theory of $\U_\epsilon(\Glie)$ is known to be closely related to that of   Lie algebras in positive characteristic \cite{ajs} (see \cite{bmr} for recent results in positive characteristic, and \cite{bak} for related results on quantum groups at roots of unity). Recently the tensor product decomposition numbers of simple representations in the root of unity case have been studied in \cite{dcprr}.

Affine Kac-Moody algebras $\hat{\Glie}$ are infinite dimensional analogues of semi-simple Lie algebras $\Glie$ \cite{kac}. They are also central extensions of loop algebras on $\Glie$. They and their quantum counterparts $\U_q(\hat{\Glie})$, called quantum affine algebras,
play an instrumental role in such areas of mathematical physics as conformal field theory and integrable quantum field theories.
There are two important classes of irreducible representations
of $\hat{\Glie}$ and of $\U_q(\hat{\Glie})$ for generic $q$.
One class is highest weight representations and the other is
finite dimensional representations. While the highest weight  representations of $\U_q(\hat{\Glie})$ are similar to those of
$\hat{\Glie}$, finite dimensional representations are quite
different, see for example \cite{cp1} and references
therein.

Also reduction of finite dimensional modules at roots  of 1 have been studied, see for example \cite{cp, fm, h, n2} and references
therein.

In this paper we focus on the specialization of $\U_q(\hat{\Glie})$ to roots of unity with
the large center. In contrast with finite dimensional Lie algebras such specializations are not finite modules over their center. The center is an infinitely generated commutative algebra (a polynomial ring with some generators inverted). Finite dimensional representations of $\U_q(\hat{\Glie})$ have been studied in \cite{bk}. These  representations have trivial central charge.  In this sense they are analogues
of evaluation representations of loop groups. Here we will focus on representations with  generic  central charge. (see section \ref{defqaa} for a precise definition of what this means).

As for finite dimensional Lie algebras, the representation theory at roots of unity is closely related to the geometry of the corresponding Lie group, which is in our situation a certain
extension of the loop group associated to ${\Glie}$. There are different topological versions of loop groups. The main purpose of this paper is the introduction of the version best suited for our goals.

The main results of this paper can be summarized as follows:

\begin{itemize}

\item First we introduce an analytic version of the loop group as the group of germs of holomorphic maps on the punctured disc. This group has the Riemann-Hilbert factorization (from Birkhoff factorization Theorem), it contains rational loops and admits central extension and the extension by the natural action of $\CC^*$. This Lie group has a natural Poisson Lie structure by the double construction.

\item We prove that the center of the specialization of $\U_q(\hat{\Glie})$ to a root of unity is naturally isomorphic as a Hopf Poisson algebra to a   ring of functions on   the Poisson dual of the analytic loop group
(more precisely, as in the finite type case, we do not realize all the center but a large Hopf Poisson subalgebra over which the center is finite). 
This geometric realization is first done for $GL_1$ and affine $GL_2$. For other affine Lie algebras it follows by a reduction to rank 2 argument.

\item We prove that the symplectic leaves are parameterized by certain $G$-bundles on an elliptic curve. The module of the elliptic curve is determined by the central charge of the quantum affine algebra.

\end{itemize}

Although in the present paper we focus on the Poisson structure
of the center of $\U_q(\hat{\Glie})$ at roots of unity and on the geometry of the relevant loop groups itself, we have in mind the study of irreducible representations of these algebras with non-trivial corresponding central charge. Such representations are parameterized by finite covers of equivalence classes of holomorphic $G$-bundles over elliptic curves.

To sum up, we have a correspondence between the following various objects :

\begin{equation*}
\begin{split}
&\text{Isomorphism classes of $G$-bundle on an elliptic curve,}
\\ \leftrightarrow &\text{\ Equivalence classes of $q$-difference equations,}
\\ \leftrightarrow &\text{\ Twisted conjugation classes in loop groups,}
\\ \leftrightarrow &\text{\ Symplectic leaves in loop groups,}
\\ \leftrightarrow &\text{\ Equivalence classes of central characters of $\U_\epsilon(\hat{\Glie})$.}
\end{split}
\end{equation*}

Note that central characters of $\U_\epsilon(\hat{\Glie})$ classify irreducible representations up to a finite cover. As in the finite dimensional case we expect that all irreducible representations for generic central characters will have the same graded dimensions.

As for quantum groups of finite type, one should expect that the study of quantum affine algebras at roots of unity will give an insight to the structure of irreducible representations of affine Lie algebras in positive characteristic.

The present paper is organized as follows : in Section \ref{manin} we give reminder on quantum algebras of finite type at roots of unity, Manin triples, Birkhoff factorizations, twisted conjugation orbits and elliptic curves. In Section \ref{defqaa} we recall the definition of quantum affine algebras with central charge and we prove the existence of a Frobenius isomorphism for the specializations at roots of unity (Proposition \ref{frgl} and \ref{frg}). In Section \ref{defian} the notion analytic loop, which is the main object studied in the present paper, is introduced with its Poisson group structure. We study a Riemann-Hilbert factorization (Theorem \ref{decompa}). In Section \ref{gl1} we start with the $GL(1)$-case which is relevant for our study as it involves the quantum Heisenberg algebra $\U_q(\hat{\mathfrak{gl}}_1)$.  In particular the elliptic curve $\mathcal{E}$ already appears in this case. The main result (Theorem \ref{isom1}) identifies the Hopf Poisson structure with the semi-classical limit of the center of $\U_q(\hat{\mathfrak{gl}_1})$. The symplectic leaves are described in Proposition \ref{sympl}. The next step is the $GL(2)$ (and $SL(2)$) situation studied in Section \ref{sldeux}. We study a Riemann-Hilbert factorization of the analytic loop group (Theorem \ref{decompb}). In particular we see how the relation between Drinfeld generators appear and the main result is Theorem \ref{isomt} where the Hopf Poisson structure is described. The holomorphic symplectic leaves are parameterized in Theorem \ref{sl}. This case is crucial to study the general case in Section \ref{gent} (Theorem \ref{isomg} and \ref{gsl}). In Section \ref{conc} the applications to representation theory and further projects are discussed.

{\bf Acknowledgments :} The authors would like to thank L. Di Vizio, B. Enriquez, V. Fock and E. Frenkel for useful discussions. The first author would like to thank ENS-Paris and Universit\'e de Versailles, the second author would like to thank the CTQM in Aarhus, the U.C. Berkeley and the University La Sapienza for hospitality.

\section{Some basic notions}\label{manin}

In this section we start with some reminders, first of the now standard theory of quantum algebras of finite type at roots of unity, and then of results which will be used in the paper.

\subsection{Quantum algebras of finite type at roots of unity}\label{finite}

Let $\Glie$ be finite dimensional semi-simple Lie algebra and let us consider the quantum group $\U_{\epsilon}(\Glie)$ (quantum algebra of finite type) specialized at a root of unity $\epsilon$ in the sense of De Concini-Kac (we refer to \cite{dcp} for details). Then $\U_\epsilon(\Glie)$ has a large center $Z_\epsilon(\Glie)$ and is finite over $Z_\epsilon(\Glie)$. This is a fundamental property of quantum groups at roots of unity which is not satisfied for a generic parameter of quantization and it allows to use the classical theory of algebras finite over their center. Moreover as we have a large commutative algebra $Z_\epsilon(\Glie)$, geometry comes naturally into the picture. By the Schur Lemma, we have the central character map :
$$\chi : \text{Spec}(\U_\epsilon(\Glie))\rightarrow \text{Spec}(Z_\epsilon(\Glie))$$
where $\text{Spec}(\U_\epsilon(\Glie))$ is the set of simple representations of $\U_\epsilon(\Glie)$ (they are finite dimensional of bounded dimension as $\U_\epsilon(\Glie)$ is finite over its center). Moreover $\chi$ is surjective, has finite fibers and is bijective on an Zariski dense open subset of $\text{Spec}(Z_\epsilon(\Glie))$. In particular we have a parametrization of generic simple representations by $\text{Spec}(Z_\epsilon(\Glie))$.

An important point is that $Z_{\epsilon}(\Glie)$ contains a Hopf subalgebra $Z$ of $\U_\epsilon(\Glie)$ with respect to which $\U_\epsilon(\Glie)$ is a free module of rank $\ell^{\text{dim}\Glie}$.
$Z$ inherits a Poisson algebra structure from the specialization (semi-classical limit) of the quantum group with generic quantization parameter and $\text{Spec}(Z)$ has a Poisson Lie group structure. $\text{Spec}(Z)$ is then identified, using the double construction, to the Poisson dual group to   $G$. The symplectic leaves of $\text{Spec}(Z)$ corresponds on one hand to conjugacy classes in $G$, and on the other hand via $\chi$ to equivalence classes of simple representations of $\U_\epsilon(\Glie)$ (the equivalence for a group of automorphism of $\U_\epsilon(\Glie)$).

A remarkable point in this theory is the relevance of the classical Lie group $G$ and its conjugacy classes for the representation theory of a quantum algebra, and how Poisson geometry comes into the picture (see also \cite{bo} for a relationship with isomonodromic deformations of irregular connections on $G$-bundles over the disc). We have a correspondence between the various objects :

\begin{equation*}
\begin{split}
  &\text{\ Conjugacy classes in $G$},
\\ \leftrightarrow &\text{\ Symplectic leaves in $\text{Spec}(Z)$},
\\ \leftrightarrow &\text{\ Equivalence classes of central characters for $\U_\epsilon(\Glie)$},
\\ \leftrightarrow &\text{\ Equivalence classes of generic representations of $\U_\epsilon(\Glie)$}.
\end{split}
\end{equation*}

\noindent This is a very interesting example of interactions between various geometric objects and algebraic representation theory.
One of the aims of the present paper is to extend this picture to the affine case. We will see that other geometric objects come into the picture.

\subsection{Manin triple and Poisson Lie groups}\label{ph}

We give some reminders (see \cite{dcp} for details) about Manin triples and the corresponding Poisson Lie groups.

Consider $(G,H,K)$ Lie groups with $H,K\subset G$  and suppose that we have a Manin triple for the corresponding Lie algebras $(\Glie, \Hlie,\mathfrak{K})$. This means  that we have a   decomposition $\Glie = \Hlie \oplus \mathfrak{K}$ as vector spaces and we have a non-degenerated, symmetric and invariant bilinear form $(,)$ such that $\Hlie$, $\mathfrak{K}$ are maximal isotropic. There is an induced Poisson group structure on $H$.   Consider $\pi : \Glie \rightarrow \Hlie$ the projection with kernel $\mathfrak{K}$, and for $h\in H$ consider $\pi^h : \Glie\rightarrow \Hlie$ defined by $\pi^h = Ad(h^{-1})\circ \pi\circ Ad(h)$. For $x,y\in \mathfrak{K}$, we set $<x,y>_h = (\pi^h(x),y)$. Then for $f_1,f_2 $ functions on $H$, the bracket $\{f_1,f_2\}$ is defined for $h\in H$ by :
$$\{f_1, f_2\}(h) = < dl_h^*(df_1(h)), dl_h^*(df_2(h))>_h,$$
where $l_h : H\rightarrow H$ defined by $l_h(g) = hg$ gives an isomorphism
$$dl_h^* : (T_h H)^*\rightarrow \Hlie^* = \mathfrak{K}.$$
\begin{rem} In the same way for $\alpha, \beta$ differential forms we can define the function $\{\alpha,\beta\} $ by the  formula $\{\alpha,\beta\}(h) = < dl_h^*(\alpha(h)), dl_h^*(\beta(h))>_h$.\end{rem}

\noindent Let $p:H\rightarrow K \setminus G$ be the restriction to $H$ of the projection to $K\setminus G$. We suppose that $G,K,H$ are connected. Then :
\begin{prop}\label{symp}
The symplectic leaves in $H$ are the connected components of the preimages under $p$ of the $K$ orbits in $K\setminus G$.
\end{prop}

If the triple $(G,H,K)$ is algebraic, namely $G$ is an algebraic group and $H$ and $K$ are algebraic subgroups in $G$, then the Poisson bracket of two regular functions on $H$ is again a regular function and the coordinate ring $\mathbb C[H]$ becomes a Poisson Hopf algebra. This means that if we denote by $\Delta:\mathbb C[H]\to\mathbb C[H]\otimes \mathbb C[H]$ the comultiplication, by $S:\mathbb C[H]\to\mathbb C[H]$ the antipode
$$\Delta(\{f,g\})=\{\Delta f,\Delta g\}$$
with $\{f_1\otimes f_2,g_1\otimes g_2\}:=f_1g_1\otimes \{f_2,g_2\}+ \{f_1,g_1\}\otimes f_2g_2$ for $f_1,g_1,f_2,g_2\in \mathbb C[H]$.

In what follows in order to compare different Poisson structures we are going to use the following simple facts which we recall here for convenience.

\begin{lem}\label{inv} Let $H$ be an algebraic group. Given  $a\in \mathbb C[H]$, the left invariant differential form coinciding with $da$ in $1$ equals $\sum S(a_{(1)}) da_{(2)}$ (here we use the Sweedler notations $\Delta(a) = \sum a_{(1)}\otimes a_{(2)}$).
\end{lem}

\begin{lem}\label{isom} Let $R,R'$ be the  two Poisson Hopf algebras  and $\Psi : R\rightarrow R'$ a ring isomorphism. Assume that are given elements $r_i\in R$ which generate $R$ as a Poisson algebra and such that :

$\forall r\in R$, we have $\Psi(\{r_i,r\}) = \{\Psi(r_i),\Psi(r)\}$,

$(\Psi\otimes \Psi)(\Delta(r_i)) = \Delta (\Psi(r_i))$,

$\Psi(S(r_i)) = S(\Psi(r_i))$.

\noindent Then $\Psi$ is an isomorphism of Poisson Hopf algebras.
\end{lem}

\subsection{Birkhoff factorization}

Let $S_1 = \{z\in\CC||z|=1\}$ be the unit circle. Consider the covering $(\mathcal{D}^+, \mathcal{D}^-)$ of the Riemann sphere $\mathbb{P}_1(\CC)$ defined by
$$\mathcal{D}^+ = \{z\in\CC| |z|\leq 1\}\text{ , }\mathcal{D}^- = \{z\in\CC | |z|\geq 1\}\cup\{\infty\}.$$
We also denote $\mathbb{P}_1(\CC)^* = \mathbb{P}_1(\CC)\setminus\{0\}$ and for $R > 0$, $\mathcal{D}^R = \{z\in\CC||z|\leq R\}$ and  $(\mathcal{D}^R)^* = \mathcal{D}^R \setminus \{0\}$.

We recall that the Lie group $GL_n(\CC)$ is connected and non simply-connected with fundamental group $\ZZ$.

Let $G$ be a finite dimensional complex algebraic group and $D$ a Cartan subgroup. $D$ is a maximal torus of dimension $n$ and so for $a_1,\cdots , a_n\in\ZZ$, choosing coordinates in $D$, we can consider the homomorphism $\lambda_{(a_1,\cdots,a_n)} : S_1\rightarrow D$ given by coordinates $(z^{a_1},\cdots,z^{a_n})$. In the case of $GL_n(\CC)$ we get maps $z\mapsto\text{diag}(z^{a_1},\cdots, z^{a_n})$.

From a theorem of Bikhoff \cite{bi1, bi2} we have (we state the version of \cite[Theorem 8.1.2]{ps}; it is explained in \cite[Section 8]{ps} that it holds for general $G$) :

\begin{thm}\label{birk} Let $\gamma : S_1\rightarrow G$ be a smooth map. Then $\gamma$ can be factorized in the form :
$$\gamma = \gamma_- \lambda \gamma_+$$
where $\gamma_\pm$ can be holomorphicaly extended to $\mathcal{D}^\pm$ and $\lambda = \lambda_{(a_1,\cdots,a_n)}$ (where $a_i\in\ZZ$). $\lambda$ is uniquely determined up to a permutation of the $a_i$. If $\lambda = 1$, the decomposition is unique if we assume that $\gamma_-(\infty) = 1$.
\end{thm}

This Theorem is a Riemann-Hilbert factorization result and will be used to prove other similar factorizations in what follows. It has many applications, as for example recently in \cite{cm}.

\subsection{Twisted conjugacy classes and elliptic curves}

Let $G$ be a connected complex
Lie group and $G(\CC^*)_{hol}$ be the group of all holomorphic maps $a: \CC^*\rightarrow G$ (possibly with
an essential singularity at $z = 0$). Fix $\Gamma\in\CC^*$ with $|\Gamma| \neq 1$. Set $\Theta=\Gamma^4$. The group $G(\CC^*)_{hol}$
is stable for the change of variable $g(z)\mapsto g(z\Theta)$ (the reason of this strange normalization will appear later on). So we can define the twisted conjugation of $G(\CC^*)_{hol}$ on itself by :
$$g(z) . a(z) = g(\Theta z)a(z)g(z)^{-1}.$$
Consider the elliptic curve :
$$\mathcal{E} = \CC^*/(z\sim \Theta z).$$
As explained in \cite{ef, bg}, Looijenga (unpublished work) proved the following beautiful parametrization of twisted conjugation classes :
\begin{thm}\label{hlg} There is a natural bijection between the set of all twisted conjugacy
classes in $G(\CC^*)_{hol}$ and the set of isomorphism classes of arbitrary holomorphic $G$-bundles
on $\mathcal{E}$.
\end{thm}
The bijection is constructed by associating to an element $a(z)\in G(\CC^*)_{hol}$ the trivial holomorphic $G$-bundle
$\CC^*\times G \rightarrow \CC^*$ on $\CC^*$ with $\Theta^{\ZZ}$-equivariant structure given by the action $(z,g) \mapsto (\Theta z, a(z)g)$.

In fact as mentioned in \cite{bg}, this result is directly related to a classification result of $q$-difference equations which is a more natural context for this parameterization.

Suppose that $G(\CC^*)_{hol}$ acts on a space $V$ over the field of holomorphic maps $\CC^*\rightarrow \CC^*$. Then for $A(z)\in G(\CC^*)_{hol}$ we consider the $q$-difference equation $X(\Theta z) = A(z) X(z)$ where $X(z)\in V$ is unknown. Then $X(z)$ is a solution of the equation if and only if $g(z)X(z)$ is a solution of the equation $Y(\Theta z) = (g(\Theta z)A(z)g(z)^{-1})Y(z)$. So the classification of twisted conjugacy classes is equivalent to the classification of the classes of these $q$-difference equations  for the transformation described above.

\begin{rem}  A result similar to Theorem \ref{hlg} for the formal loop group is proved in \cite{bg}. We refer to \cite{a, fmw, l} and references therein for vector and $G$-bundles on elliptic curves. We refer to \cite{et} to other connection between quantum affine algebras and vector bundles on elliptic curves, and to \cite{beg} for another connection to representation theory.\end{rem}

\section{Quantum affine algebras and their specializations}\label{defqaa}

We recall the definition of quantum affine algebras with central charge and we prove the
existence of a Frobenius isomorphism for the specializations at roots of unity
(Proposition \ref{frgl} and \ref{frg}). A particular attention is given to the quantum Heisenberg algebra $\U_q(\hat{\mathfrak{gl}}_1)$.

For $l\in\ZZ, r\geq 0, m\geq m'\geq 0$ we define in $\ZZ[q^{\pm}]$ :
$$[l]_q=\frac{q^l-q^{-l}}{q-q^{-1}}\text{ , }[r]_q!=[r]_q[r-1]_q\cdots [1]_q\text{ ,
}\begin{bmatrix}m\\m'\end{bmatrix}_q=\frac{[m]_q!}{[m-m']_q![m']_q!}.$$

\subsection{The quantum Heisenberg algebra $\U_q(\hat{\mathfrak{gl}}_1)$}

Let us start with the simplest example of quantum affine algebra, that is to say the quantum Heisenberg algebra $\U_q(\hat{\mathfrak{gl}}_1)$. It is a particular example as $\mathfrak{gl}_1$ is not a simple Lie algebra, but we will see in the following that it is of particular importance for the purposes of the present paper.

\begin{defi} $\U_q(\hat{\mathfrak{gl}_1})$ is the $\CC(q)$-algebra with generators $h_m$ ($m\in\ZZ-\{0\}$), $\Lambda^{\pm 1}$, central elements $\Gamma^{\pm 1}$, $k$, and relations :
$$[h_m,h_{-m'}] = \delta_{m,-m'}\frac{1}{m}[m]_q \frac{\Gamma^{2m} - \Gamma^{-2m}}{q - q^{-1}}\text{ , }h_m  \Lambda= q^{2m} \Lambda h_m.$$
\end{defi}
A priori $k$ does not have an important role in the structure of the algebra as it is central and does not appear in the relations, but its importance will appear in Section \ref{pb}.

We have a Hopf algebra structure on $\U_q(\hat{\mathfrak{gl}_1})$ given by :
$$\Delta(h_m) = h_m \otimes \Gamma^{-|m|} + \Gamma^{|m|}\otimes h_m,$$
$$\Delta(\Gamma) = \Gamma\otimes \Gamma\text{ , }\Delta(\Lambda) = \Lambda\otimes \Lambda \text{ , }\Delta(k) = k\otimes k,$$
$$S(h_m) = -h_m\text{ , }S(\Gamma) = \Gamma^{-1}\text{ , }S(\Lambda) = \Lambda^{-1}\text{ , }S(k) = k^{-1},$$
$$\varepsilon(h_m) = 0\text{ , }\varepsilon(\Gamma) = \varepsilon(\Lambda) = \varepsilon(k) = 1.$$

Let $\epsilon$ be a primitive $\ell$-root of unity ($\ell\geq 3$ is odd). Let us consider the specialization $\U_\epsilon(\hat{\mathfrak{gl}}_1)$ of $\U_q(\hat{\mathfrak{gl}}_1)$ at $q = \epsilon$. In order to define it  we consider the $\CC[q^{\pm 1}]$-subalgebra generated by the $(q - q^{-1})h_m$, $\Lambda^{\pm 1}$, $\Gamma^{\pm 1}$ and we quotient by the ideal generated by $(q - \epsilon)$.

\begin{prop} The center of $\mathcal{U}_{\epsilon}(\hat{\mathfrak{gl}}_1)$ is the algebra $Z_{\epsilon}(\hat{\mathfrak{gl}}_1)$ generated by the elements $h_{r\ell}$, $k^{\pm 1}$, $\Gamma^{\pm 1}$, $\Lambda^{\pm \ell}$ ($r\in\ZZ-\{0\}$).\end{prop}

\demo The above elements are clearly central. There is a PBW theorem so we get a basis $\prod_{m\in\mathbb Z \setminus \ell\ZZ}h_m^{s_m} \Lambda^t$,  $0\leq t\leq \ell-1$, $s_m$ a non negative integer, for $\mathcal{U}_{\epsilon}(\hat{\mathfrak{gl}}_1)$ over $Z_{\epsilon}(\hat{\mathfrak{gl}}_1)$. Then take a linear combination of these basis elements and take the lexicographically largest sequence $(t,s_{-1},s_{1}, \ldots ,
s_{1-\ell},s_{\ell-1},s_{-\ell-1},\ldots )$ such that the corresponding monomial appears in the linear combination with non zero coefficient. If $t>0$ it suffices to commute with $h_1$ to show that
the element cannot be in the center. Otherwise we commute with $h_m$ where $-m$ is the first index such that $s_{-m}>0$. Also in this case we are done. So we remain with the zero sequence which means we are in   $Z_{\epsilon}(\hat{\mathfrak{gl}}_1)$.\qed

Note that $\mathcal{U}_{\epsilon}(\hat{\mathfrak{gl}}_1)$ is not a finite module over its center. We denote by $Z_{\epsilon}$ the subalgebra of the center generated by the $h_{r\ell}$, $k^{\pm 1}$, $\Gamma^{\pm \ell}$, $\Lambda^{\pm \ell}$ ($r\in\ZZ-\{0\}$). The center is free of rank  $\ell$ over $Z_{\epsilon}$ and clearly $Z_{\epsilon}$ is a Hopf subalgebra of $\mathcal{U}_{\epsilon}(\hat{\mathfrak{gl}}_1)$.

As in \cite{dckp, r, bk}, we can define a Poisson structure on $Z_\epsilon(\hat{\mathfrak{gl}}_1)$. For $x,y\in Z_\epsilon(\hat{\mathfrak{gl}}_1)$, we define $\{x,y\} = [\tilde{x},\tilde{y}]/(\ell(q^\ell - q^{-\ell})) \text{Mod}(q - \epsilon)$ where $\tilde{x},\tilde{y}$ are respective representative of $x,y$ in $\U_q(\hat{\mathfrak{gl}}_1)$. Clearly $Z_{\epsilon}$ is a Poisson subalgebra and becomes a Poisson Hopf algebra.

In fact the specialization also makes sense for $\epsilon = 1$ : we get a commutative algebra $Z_1=\U_1(\hat{\Glie})$. $Z_1$ is a Poisson Hopf algebra as above (in this case $\ell = 1$).

By using a Frobenius isomorphism, we can see that the structure is independent on $\ell$ :

\begin{prop}\label{frgl} There is a Hopf Poisson algebra isomorphism $Fr : Z_1\rightarrow Z_\epsilon $ defined by
$$Fr((q -  q^{-1})h_m) = \ell(q - q^{-1})h_{m\ell}\text{ , }Fr(\Gamma) = \Gamma^{\ell}\text{ , }Fr(\Lambda) = \Lambda^{\ell}.$$
\end{prop}

\demo The relation
$$[\ell(q - q^{-1})h_{m\ell},\ell (q-q^{-1})h_{-m'\ell}] = \ell(q - q^{-1})\delta_{m\ell,-m'\ell}\frac{1}{m}[m\ell]_q(\Gamma^{2m\ell} - \Gamma^{-2m\ell})$$
gives
$$\{\ell(q - q^{-1})h_{m\ell},\ell(q-q^{-1})h_{-m'\ell}\} = \delta_{m,-m'}\frac{1}{m}[m\ell]_q(Fr(\Gamma^{2m}) - Fr(\Gamma^{-2m}))$$
$$= \delta_{m,-m'}(Fr(\Gamma^{2m}) - Fr(\Gamma^{-2m})).$$
This coincides with the relation in $Z_1$ :
$$\{(q - q^{-1})h_m , (q - q^{-1})h_{-m'}\} = (\Gamma^{2m} - \Gamma^{-2m}) \delta_{m,-m'}.$$
The relation $h_{\ell m}\Lambda^\ell = q^{2m\ell^2} \Lambda^\ell h_{\ell m}$ gives
$$\{\ell(q - q^{-1})h_{m\ell},\Lambda^\ell\} = \frac{q^{2m\ell^2} - 1}{\ell(q^\ell - q^{-\ell})}\Lambda^\ell(\ell(q - q^{-1})h_{m\ell}) = m \Lambda^\ell(\ell(q - q^{-1})h_{m\ell}).$$
This coincides with the relation in $Z
_1$ :
$$\{(q - q^{-1})h_m, \Lambda\} = m\Lambda (q - q^{-1})h_m.$$
The Hopf algebra structure is clearly preserved by the Frobenius map.
\qed

\subsection{Definition}

Let $\Glie$ be a simple Lie algebra, $C=(C_{i,j})_{1\leq i,j\leq n}$ its Cartan matrix and $\Hlie\subset \Glie$ is a Cartan subalgebra. We set $I=\{1,\dots,n\}$. $C$ is symmetrizable, that is to say that there is a matrix $D=\text{diag}(r_1,\dots,r_n)$ ($r_i\in\NN^*$)\label{ri} such that $B=DC$\label{symcar} is symmetric.
$\Pi=\{\alpha_1,\dots,\alpha_n\}\subset \Hlie^*$ is set of the
simple roots.

We consider the quantum affine algebra associated to the affine Kac-Moody algebra $\hat{\Glie}$. We use the Drinfeld presentation of the algebra \cite{Dri2, b} :

\begin{defi} $\U_q(\hat{\Glie})$ is the $\CC(q)$-algebra with generators $h_{i,r}$ ($i\in I, r\in\ZZ-\{0\}$), $x_{i,m}^{\pm}$ ($i\in I, m\in\ZZ$), $k_i^{\pm 1}$, $\Lambda^{\pm 1}$, central elements $\Gamma^{\pm 1}$, and relations :
$$[h_{i,m},h_{j,-m'}] = \delta_{m,-m'}\frac{1}{m}[mB_{i,j}]_q \frac{\Gamma^{2m} - \Gamma^{-2m}}{q - q^{-1}},$$
$$h_m\Lambda = q^{2m} \Lambda h_m\text{ , }x_m^\pm \Lambda = q^{2m}\Lambda x_m^\pm \text{ , }[k_i,h_m] =[\Lambda,k_i]=[\Lambda,h_{i,m}]= 0,$$
$$k_ix_{j,m}^{\pm}k_i^{-1}= q^{\pm B_{i,j}}x_{j,m}^{\pm},$$
$$[h_{i,m},x_{j,m'}^{\pm}]=\pm \frac{1}{m}[mB_{i,j}]_q \Gamma^{\mp\mid m\mid} x_{j,m+m'}^{\pm},$$
$$x_{i,m+1}^{\pm}x_{j,m'}^{\pm}- q^{\pm B_{i,j}}x_{j,m'}^{\pm}x_{i,m+1}^{\pm}= q^{\pm B_{i,j}}x_{i,m}^{\pm}x_{j,m'+1}^{\pm}-x_{j,m'+1}^{\pm}x_{i,m}^{\pm},$$
$$[x_{i,m}^+,x_{j,m'}^-]= \delta_{i,j}\frac{\Gamma^{m-m'}\phi^+_{i,m+m'}- \Gamma^{m' - m}\phi^-_{i,m+m'}}{q^{r_i} - q^{-r_i}},$$
$$\underset{\pi\in\Sigma_s}{\sum}\underset{k=0\cdots s}{\sum}(-1)^k\begin{bmatrix}s\\k\end{bmatrix}_{q^{r_i}}x_{i,r_{\pi(1)}}^{\pm}\cdots x_{i,r_{\pi(k)}}^{\pm}x_{j,r'}^{\pm}x_{i,r_{\pi(k+1)}}^{\pm}\cdots x_{i,r_{\pi(s)}}^{\pm}=0,$$
where the last relation holds for all $i\neq j$, $s=1-C_{i,j}$, all sequences of integers $r_1,\cdots ,r_s$. $\Sigma_s$ is
the symmetric group on $s$ letters. The $\phi_{i,\pm m}^{\pm}\in \U_q(\hat{\Glie})$ are defined for
$m\geq 0$ by the formal power series
$$\phi_i^{\pm}(u)=\underset{m\geq 0}{\sum}\phi_{i,\pm m}^{\pm}u^{\pm m}=k_i^{\pm 1} \text{exp}(\pm \underset{m'\geq 1}{\sum}(q - q^{-1})h_{i,\pm m'}u^{\pm m'}),$$
and we set $\phi_{\pm m}^{\pm}=0$ for $m<0$.
\end{defi}
The algebra $\U_q(\hat{\Glie})$ has a Hopf algebra structure which is defined in terms of the Drinfeld-Jimbo generators (see \cite{cp1}).

Let $\epsilon$ be a primitive $l$-root of unity ($l\geq 3$ is odd and prime with the $r_j$). Let us consider the specialization $\U_\epsilon(\hat{\Glie})$ of $\U_q(\hat{\Glie})$ at $q = \epsilon$ : we consider the $\CC[q^{\pm 1}]$-subalgebra generated by the $(q - q^{-1})h_m$, $(q^{r_i} - q^{-r_i})x_{i,m}^\pm$, $\Lambda^{\pm 1}$, $\Gamma^m$ and we quotient by the ideal generated by $(q - \epsilon)$.

\subsection{The center and its Poisson structure}

Consider the set of real roots of the affine Lie algebra $\tilde{\Delta}^{re} = \Delta + \ZZ\delta$ where $\Delta$ is the set of roots of $\Glie$ and $\delta$ generates the imaginary roots. In \cite{b, da0} root vectors $x_\beta\in\U_q(\hat{\Glie})$ ($\beta\in \tilde{\Delta}^{re}$) are considered. This includes the $x_{i,m}^\pm = x_{\pm\alpha_i + m\delta}$.

Under some additional condition (see \cite[Corollary 2.1]{bk}), the center of $\mathcal{U}_{\epsilon}(\hat{\Glie})$ was calculated in \cite{da} and \cite[Proposition 2.3, Remark 2.3]{bk} :

\begin{thm}\label{bkcenter} $Z_\epsilon(\hat{\Glie})$ is generated by the $(x_\beta)^{\ell}$, $h_{r\ell}$, $k_i^{\pm \ell}$, $\Gamma^{\pm 1}$, $\Lambda^{\pm \ell}$ ($\beta\in \tilde{\Delta}^{re}, r\in\ZZ-\{0\}$).\end{thm}

Remarks :

- $Z_\epsilon(\hat{sl_2})$ is generated by $(x_{i,r}^\pm)^{\ell}$, $h_{r\ell}$, $k_i^{\pm \ell}$, $\Gamma^{\pm 1}$, $\Lambda^{\pm \ell}$ ($r\in\ZZ-\{0\}$, $m\in\ZZ$).

- If $E_i$, $F_i$, $K_i^{\pm 1}$ ($0\leq i\leq n$) denote the Drinfeld-Jimbo generators of $\U_\epsilon(\hat{\Glie})$  which give the Chevalley presentation of $\U_\epsilon(\hat{\Glie})$, then $E_i^\ell$, $F_i^\ell$, $K_i^\ell$ are in the center.

- $\mathcal{U}_{\epsilon}(\hat{\Glie})$ is not of finite rank over its center.

\noindent As for $\hat{\mathfrak{gl}}_1$, we define $Z_\epsilon$ as the subalgebra of the center with the same generators except $\Gamma^{\pm 2\ell}$ instead of $\Gamma^{\pm 1}$ and we have a Poisson structure on $Z_\epsilon(\hat{\Glie})$ such that $Z_{\epsilon}$ is a Poisson subalgebra of $Z_\epsilon(\hat{\Glie})$  and a sub Hopf algebra of $\U_\epsilon(\hat{\Glie})$.

\noindent The specialization also makes sense for $\epsilon = 1$ and we get a commutative Poisson Hopf algebra $Z_1:=\U_1(\hat{\Glie})$. We have
\begin{prop}\label{frg} There is a Hopf Poisson algebra isomorphism $Fr : Z_1\rightarrow Z_\epsilon$ defined by $$Fr((q^{r_i} - q^{-r_i})x_{i,m}^\pm) = (x_{i,m}^{\pm})^\ell\text{ , }Fr((q -  q^{-1})h_{i,m}) = \ell(q - q^{-1})h_{i,m\ell},$$
$$Fr(\Gamma) = \Gamma^{\ell}\text{ , }Fr(\Lambda) = \Lambda^{\ell}.$$
\end{prop}

\demo When $\Gamma = 1$, the result is proved in \cite[Section 3.2]{bk} by using results from \cite{l2} for the subalgebra without $\Lambda$. We work with the Drinfeld generators of the Poisson algebra. We can check as for $\hat{\mathfrak{gl}}_1$ that the result is preserved with the additional generator $\Lambda$. With the $\Gamma$ it is a little bit more complicated as it appears in various relations between generators. The relations between the $h_{i,r}$, which are trivial in the $\Gamma = 1$ case, give rise to quantum Heisenberg algebras and so can be treated as for $\hat{\mathfrak{gl}}_1$. The relations between the $x_{i,m}^+$ (resp. $x_{i,m}^-$) are not deformed.

For the relations between the $h_{i,m}$ and the $x_{j,m'}^\pm$, if $m > 0$ it suffices to replace $x_{j,M}^\pm$ by $\Gamma^{\mp M}x_{j,M}^\pm$, that is to say :
$$[h_{i,m},(x_{j,m'}^\pm \Gamma^{\mp m'})] = \pm \frac{1}{m}[mB_{i,j}]_\epsilon (\Gamma^{\mp (m + m')}x_{j,m + m'}^\pm).$$
Then we can use the result with trivial central charge as the relation between the $x_{j,m}^+$ is unchanged by the replacement. The case $m < 0$ is treated in the same way by using $\Gamma^{\pm M}x_{j,M}^\pm$ instead of $x_{j,M}^\pm$.

For the relations between the $x_{i,m}^+$ and the $x_{i,m'}^-$, we study the case $m+m'\geq 0$ (the case $m+m'\leq 0$ is treated in an similar way). Then we have to use relations of the form
$$[x_{i,M}^+,x_{i,M'}^-] = \Gamma^{M - M'}\phi_{i,M + M'}^+/(q_i - q_i^{-1})\text{ , }[h_{i,M},x_{i,M'}^\pm] = \pm[2M]_q\Gamma^{\mp M}x_{i,M + M'}^\pm /M.$$
So it suffices to replace as above $x_{i,m}^\pm$ by $\Gamma^{\mp 2}x_{i,m}^\pm$ and to use the result with trivial central charge.

As for the case $\Gamma = 1$, the Hopf algebra structure is clearly preserved on the Drinfeld-Jimbo generators by the Frobenius map.
\qed

As a Poisson algebra, $Z_1$ is generated by the Chevalley generators $E_i$, $F_i$, $K_i^{\pm 1}$, $\Lambda^{\pm 1}$ ($0\leq i\leq n$) as well as  by the Drinfeld generators $x_{i,m}^\pm$, $k_i^{\pm 1}$, $\Gamma^{\pm}$, $\Lambda^{\pm}$. Since  the isomorphism $Fr$ maps the elements $E_i$, $F_i$, $K_i^{\pm 1}$, $\Lambda^{\pm 1}$ to the elements $E_i^\ell$, $F_i^\ell$, $K_i^\ell$, $\Lambda^{\pm \ell}$ ($0\leq i\leq n$), it follows that these elements generate   $Z_\epsilon$ as a Poisson algebra.

\section{Analytic loop groups and their extensions}\label{defian}

\subsection{Loop groups}
Let $G$ be  a finite dimensional complex algebraic group. We want to define a variant of the notion of the loop group $\hat {G}$ that is something which should be "morally" a group of maps $S_1\to G$.

Given a positive real number $R$ we shall denote by $\mathcal D_R^*$ the pointed disk of radius $R$ i.e.
$$\mathcal D_R^*=\{z\in\mathbb C|0<|z|\leq R\}.$$

Take  a finite dimensional  complex manifold $M$. Consider  pairs $(R,f)$ where
$f: \mathcal D_R^*\to M$  is a continuous map holomorphic in the interior of $\mathcal D_R^*$. Two such pairs  $(R,f)$ and $(R',f')$  are said to be equivalent if there exists    $R''\leq$  min $(R,R')$  such that  $f_{|\mathcal{D}_{R''}^*} = f'_{|\mathcal{D}_{R''}^*}$. An equivalence class is called a germ of a holomorphic map on the punctured disc with values in $M$. In the following we will just write $f$ for $(R,f)$  and denote by $LM$ the set of germs of   holomorphic maps on the punctured disc with values in $M$.

When the target manifold is the field of complex numbers $\mathbb C$ we set  $\mathcal L:=L{ \mathbb C}$ and we have

\begin{lem}\label{dlie}\cite{adc} There is a canonical decomposition :
$$\mathcal {L} = \mathcal{L}^+\oplus \mathcal{L}^-$$
where $\mathcal L^+$ is the space of germs of holomorphic functions around zero and $\mathcal L^-$ the space of holomorphic functions on $\mathbb P^1\setminus\{0\}$ such that $f(\infty)=0$.
\end{lem}
A topology on $\mathcal{L}$ is defined in \cite[Section 1]{adc} as the product topology of a Fr\'echet topology on $\mathcal{L}^-$ and a topology on $\mathcal{L}^+$ seen as a direct limit as $R$ goes to zero of the Banach spaces of holomorphic functions on $\mathcal{D}^R$. We get a locally convex topological vector space.
 In power series we can write $f\in \mathcal L$ as $f=\sum_{n\in\mathbb Z}f_nz^n$,
$f_n\in\mathbb C$ and $f^+=\sum_{n\geq 0}f_nz^n$, $f^-=f-f^+$.

Using this, we see that if $U\subset \mathbb C^n$ is an open set, then $LU$ is an open set in $\mathcal L^n$.  It is then not hard to see that
$LM$ has the structure of an infinite dimensional complex manifold based on $\mathcal L$.
When $G$ is a complex Lie group, we can give $LG$ the structure of an infinite dimensional complex Lie Group by defining charts as follows (see \cite{ps}).  Take a open neighborhood $A$ of the identity element  in $G$ which is homeomorphic by the exponential  map to a open neighborhood $\tilde A$ of $0$ in the Lie algebra  $\mathfrak g$ of $G$. Then we can identify $LA$ with $L\tilde A$ via the exponential map. At this point for any   $g\in LG$ we can translate $LA$ using right multiplication by $g$ to the   neighborhood $LA_g:=LA\cdot g$. The collection $\{LA_g\}$ gives the desired atlas for $LG$.
\begin{defi} The infinite dimensional complex Lie group $LG$ is the analytic loop group of the group $G$.
\end{defi}

From this it is then immediate that the Lie algebra $\mathfrak{L g}$ of $LG$ coincides with $L\mathfrak g$ with the obvious Lie algebra structure. We define as above $\mathfrak{L g}^\pm\subset \mathfrak{L g}$.

There are various reasons to choose this definition. Let us list the three which we consider relevant for this paper.

1) The group of rational functions $G(z)$ that is the group of $\mathbb C(z)$-rational points embeds in $LG$. This is clear since  any   $G$ valued rational map $\tilde g$ is defined outside a finite subset in $\mathbb C$. Hence can be restricted to a map
$\mathcal D_R^*\to G$ for a small enough $R$. By a similar argument $LG$ also contains the group of $G$ valued holomorphic maps on $\mathbb C^*$.

2)   $ \mathbb C^*$ acts  on $LG$ by rescaling. Indeed take $g\in LG$ and fix a representative $(R,f)$ for $g$. Given $\gamma\in\mathbb C^*$ define $f^{\gamma}:\mathcal D_{R/\gamma}\to G$ by $f^{\gamma}(z)=f(\gamma z)$. The equivalence class of $(R/\gamma,f^\gamma)$ depends only on $g$ and defines and element $g^\gamma\in LG$.

3) $LG$ naturally contains the subgroup $LG^+$ consisting of germs of $G$ valued holomorphic map around zero and the subgroup $LG^-$ consisting of $G$ valued holomorphic functions on $\mathbb{P}^1\setminus \{0\}$ whose Lie algebras are $\mathfrak {Lg}^+$ and $\mathfrak {Lg}^-$ respectively.

In order to justify our definition we will briefly discuss different versions of loop groups studied in the literature and we are going to explain some of their properties and point out some of their drawbacks with respect the purposes of the present paper.

\subsubsection{Formal loop group}

The formal loop group $G((z))$ consists of invertible complex valued formal Laurent series in $z$.

Formal loop groups are of fundamental importance, as for example in the geometric Langlands program \cite{f2}. However, they do not work for our purposes for several reasons. The main reason is that although at the Lie algebra level formal power Laurent series
$\Glie((z))$ factorize into $z^{-1}\Glie[z^{-1}]\oplus \Glie[[z]]$, there is no such factorization at the level of Lie groups so there is no analog of property 3).

\subsubsection{Smooth Loop group}

The smooth loop group $L(G)$ consists of smooth maps $S_1\rightarrow G$. $L(G)$ is considered in \cite{ps} (we refer to this book as well as to \cite{fa} for details about its structure of infinite dimensional Lie group).

 There is an action by rescaling of $S_1$ on $L$ and so the change of variable makes sense for $|\gamma| = 1$, but for general $\gamma\in\CC^*$ this makes no sense. So Property 2) is not satisfied.

However let us explain some properties of $L(G)$ which will be useful in the following.

Notice that an analogue of property 3) is satisfied taking as
 $ L^{\pm}$   the subgroups of those  $f\in L$ which are the boundary value of an holomorphic map $\mathcal{D}^{\pm}\rightarrow G$. For $f\in L^{\pm}$ and $z\in \mathcal{D}^{\pm}$ we still denote the extension of $f$ to $\mathcal{D}^{\pm}$ by $f(z)$.

Theorem \ref{birk}, gives a Riemann-Hilbert factorization of $f\in L$ in the form $f = f_+ \lambda f_-$ where $f_{\pm}\in L^\pm$.

\subsubsection{Holomorphic loop group}

The holomorphic loop group $G_{hol}$ consists of  holomorphic maps $\CC^*\rightarrow G$.

The holomorphic loop group does not have the drawbacks mentioned above but on the other hand  does not contain rational maps i.e. it does not satisfy our condition 1). Since rational maps are known to have particular importance in the representation theory of quantum affine algebras (see the discussions in \cite{bk} in the root of unity case, but also from the results in the restricted cases in \cite{cp} involving Drinfeld polynomials)  it is important to consider our larger group $LG$ in what follows.

\subsubsection{Riemann-Hilbert factorization}
Let $G$ be a reductive complex algebraic group.  $\Glie$ be its Lie algebra. $D$  a  maximal torus in $G$, $B^+$, $B^-$  opposite Borel subgroups with respect to a choice of positive roots so that  $D = B^+\cap B^-$. $U^\pm\subset B^\pm$ the corresponding unipotent subgroups. We set $\mathfrak{d} = \text{Lie}(D)$, $\mathfrak{b}_\pm = \text{Lie}(B^\pm)$, $\mathfrak{n}_\pm = \text{Lie}(U^\pm)$.

Passing to $LG$ we obtain inclusions of $LD$, $LB^\pm$, $LU^\pm$ into $LG$.

Let us now prove a Riemann-Hilbert factorization for $LG$.

\begin{thm}\label{decompa} An element $f\in LG$ can be written in the form $f = f_+ \lambda f_-$ where $f_{\pm}\in {LG}^\pm$, $f_-(\infty) = 1$ and $\lambda$ is a one parameter subgroup of $D$, that is the germ of a homomorphism $\lambda:\mathbb C^*\to D$.
\end{thm}

\demo There is $R > 0$ such that $f$ is defined on $(\mathcal{D}_R)^*$. Then $F(z) = f(zR)$ is in the smooth loop group $L(G)$ and can be extended to an holomorphic map $\mathcal{D}^+ - \{0\}\rightarrow G$. We decompose $F = F^+ \lambda F^-$ by using Theorem \ref{birk} for the smooth loop group where $\lambda$ is the restriction to $S^1$ of a one parameter subgroup in $D$. Then as $F^+$ {\rm and} $F$ can be extended to an holomorphic map $\mathcal{D}^+-\{0\}\rightarrow G$, $F^-$ can be extended to an holomorphic map $\mathbb{P}_1(\CC)\setminus\{0\}\rightarrow G$. Now consider $f^+(z) = F^+(zR^{-1})\lambda(R^{-1})$, $f^-(z) = F^-(zR^{-1})$ as elements of $LG$ defined on $(\mathcal{D}^R)^*$. Then in $LG$ we have $f = f^+ \lambda f^-$.
\qed

\subsection{Extensions} Let $LG^0$ be the connected component of the origin in $LG$ (remark that if $G$ is simply connected then $LG^0=LG$). Let us define a central extension $\tilde{LG}$ of $LG^0$. For each $R > 0$, by using the construction in \cite[Section 4]{ps} we have a central extension of the group of holomorphic maps $(\mathcal{D}^R)^*\rightarrow G$ (as a topological space it is a non-trivial fiber bundle over the group with the circle as fibers). The construction relies on a $2$-cocycle which does not depend on a rescaling, so the central extensions are compatible for various $R$ and we get a central extension $\tilde{LG}$ of the Lie group $LG^0$. Moreover the rescaling action $\tau : \CC^*\rightarrow \text{Aut}(LG^0)$ can be lifted to an action $\tau : \CC^*\rightarrow \text{Aut}(\tilde{LG})$.

Then we consider an extension $\hat{G}$ of $\tilde{LG}$. As a manifold  $\hat{G} = \tilde{LG}\times \CC^*$. For $f,g\in\tilde{LG}$ we put
$$(f,\Gamma)(g,\Gamma') = (\tau((\Gamma')^{-1})(f)\tau(\Gamma)(g),\Gamma\Gamma').$$
The Lie algebra of $\hat{G}$ is the following extension $\hat{\mathfrak{g}}$ of $\mathfrak{Lg}$ (analog to the extension of the standard loop algebra, see \cite{kac}). As a topological vector space $\hat{\mathfrak{g}} = \mathfrak{Lg} \oplus \CC c\oplus \CC d$. For $f,g\in\mathfrak{Lg}$, $\lambda,\lambda',\mu,\mu'\in\CC$, the bracket is given by:
$$[f + \lambda c + \mu d,g + \lambda' c + \mu' d] = [f,g]_{\mathfrak{Lg}} + \text{Res}_{z = 0}(B(g,df)) c - \mu ' z \frac{df}{dz} + \mu z \frac{dg}{dz},$$
where $B$ is an invariant bilinear form obtained from an invariant bilinear non degenerated form on $\Glie$.

We set $LG^{0,\pm} = LG^0\cap LG^\pm$ and we define in the same obvious way $\tilde{LG}^\pm, \hat{G}^\pm$.

\subsection{A Manin triple}

Starting from the Lie algebra $\hat{\mathfrak {g}}$ of the Lie group $\hat{G}$ we are going to construct a Manin triple. The definition of the triple is motivated by the Riemann-Hilbert factorization in Theorem \ref{decompa}.

We write $0^{-1} = \infty$. Let us consider the following connected Lie groups:
$$\tilde{G} = \hat{G} \times \hat{G},$$
$$H = \{((f_+,\Lambda ,\Gamma),(f_-,\Lambda^{-1},\Gamma^{-1}))\in\tilde{G}| f_\pm\in LG^\pm, f_\pm(0^{\pm 1}) \in h^{\pm 1} U^\pm, h\in D\},$$
$$K = \{(F,F)|F\in \hat{G}\}.$$
\begin{rem}\label{wdef} A priori the group $H$ can not be defined as above as the central extension of $LG^0$ is given by non-trivial fiber bundles. But it follows from the construction in \cite[6.6]{ps} the central extension is canonically trivial when restricted to $LG^{0,+}$ or to $LG^{0,-}$ (a canonical section is given by the determinant in \cite[6.6]{ps}). So the group $H$ is well-defined and the formula for $H$ given above makes sense.
\end{rem}

Notice that  $LG^{0,+}\cap LG^{0,-}$ consists of holomorphic maps $\mathbb{P}_1(\CC)\rightarrow G$ and so it reduces to $G$. In particular $H\cap K \cong (\mathbb Z/2\mathbb Z)^{n+2}$. Let us consider the respective Lie algebras of $\tilde{G}$, $H$ and $K$ :
$$\tilde{\mathfrak{g}} = \hat{\mathfrak{g}}\times \hat{\mathfrak{g}},$$
$$\Hlie = \{(a_+ + \lambda c + \mu d)\oplus (a_- - \lambda c - \mu d)\in \tilde{\mathfrak{g}} | a_\pm\in \mathfrak{Lg}^\pm,a_\pm(0^{\pm 1}) \in \pm d + \mathfrak{n}_\pm, d\in\mathfrak{d} \},$$
$$\mathfrak{K} = \{g\oplus g|g\in \hat{\mathfrak{g}}\}.$$

As for usual Kac-Moody algebras \cite{kac} we have :

\begin{lem} The bilinear form on $\hat{\mathfrak{g}}$ defined by ($f,g\in\mathfrak{Lg}$, $\lambda,\mu,\lambda',\mu'\in\CC$) :
$$(f+\lambda c + \mu d,g + \lambda'c + \mu'd) = \text{Res}_{z=0}(B(f,g)z^{-1}) + \lambda\mu' + \mu\lambda',$$
is symmetric, invariant and non degenerated ($B$ is naturally extended to a bilinear form $\mathfrak{Lg}\times\mathfrak{Lg}\rightarrow\mathcal{L}$).
\end{lem}

\demo The symmetry is clear. For the invariance, as $\text{Res}_{z=0}B(g,df)$ is antisymmetric and $B$ is invariant, we have :
$$([f+\lambda c + \mu d,g + \lambda'c + \mu'd],h + \lambda''c+\mu''d)$$
$$= ([f,g]_{\mathfrak{Lg}} + \text{Res}_{z=0}(B(g,df)) c - \mu ' z \frac{df}{dz} + \mu z \frac{dg}{dz} , h + \lambda''c+\mu''d) $$
$$= \text{Res}_{z=0}(B([f,g]_{\mathfrak{Lg}},h)z^{-1})  -\mu' \text{Res}_{z=0}B(h,df) + \mu \text{Res}_{z=0}B(h,dg) + \mu'' \text{Res}_{z=0}B(g,df)$$
$$= (f+\lambda c + \mu d, [g,h]_{\mathfrak{Lg}} + \text{Res}_{z=0}B(h,dg) c - \mu'' z \frac{dg}{dz} + \mu' z \frac{dh}{dz}) $$
$$= (f+\lambda c + \mu d,[g + \lambda'c + \mu'd,h + \lambda''c+\mu''d]).$$

For the non-degeneracy, suppose that for any $g + \lambda'c + \mu'd \in \hat{\mathfrak{g}}$ we have :
$$(f+\lambda c + \mu d,g + \lambda'c + \mu'd) = 0.$$
As $\text{det}(\begin{pmatrix} 0 & 1\\ 1 & 0\end{pmatrix})\neq 0$ we have clearly $\lambda = \mu = 0$. Write $f = \sum_{n\in\mathbb Z}a_nz^n$. Then $B(a_n,b) = 0$ for any $b$ so $f = 0$.
\qed

Using the  bilinear form previously defined on $\hat{\mathfrak{g}}$, we get a bilinear form on $\tilde{\mathfrak{g}}$ setting for  $a,b,c,d\in\hat{\mathfrak{g}}$  :
$$\langle a\oplus b,c\oplus d\rangle = (a,c) - (b,d).$$

\begin{prop} $(\tilde{\mathfrak{g}}, \Hlie, \mathfrak{K})$ is a Manin triple with respect to the bilinear form $\langle-\, ,- \rangle$.\end{prop}

\demo First as $(-,-)$ is non degenerated, symmetric and invariant on $\hat{\Glie}$, $\langle -,- \rangle$ has the same properties on $\tilde{\Glie}$.

We have $\tilde{\Glie} = \Hlie\oplus\mathfrak{K}$. Indeed if $(f+\lambda c + \mu d)\oplus (g + \lambda' c + \mu' d)\in\Hlie\cap \mathfrak{K}$, then $\lambda = \lambda ' = -\lambda '$ and so $\lambda = \lambda' = 0$ and in the same way $\mu = \mu' = 0$. We also have $f = g \in \mathfrak{Lg}^+\cap \mathfrak{Lg}^-$ and so $f$ is constant. But $f(0) = g(\infty) \in (\mathfrak{n}_+ + d )\cap (\mathfrak{n}_- - d)$ and so $f = g = 0$.
For $a\oplus b\in \tilde{\Glie}$, let us write $a = a_+ + a_0 + a_- + \lambda c + \mu d$ and $b = b_+ + b_0 + b_- + \lambda' c + \mu' d$ where $a_{\pm}, b_{\pm}\in \mathfrak{Lg}^\pm$, $(a_\pm)_0 = (b_\pm)_0 = 0$ and $a_0 , b_0\in \Glie$. Let us write $a_0 = a_0^+ + a_0^0 + a_0^-$, $b_0 = b_0^+ + b_0^0 + b_0^-$ where $a_0^\pm, b_0^\pm \in \mathfrak{n}_\pm$ and $a_0^0, b_0^0\in\mathfrak{d}$. Then let :
$$c = a_- + b_+ + b_0^- + a_0^+ (a_0^0 + b_0^0 + (\lambda + \lambda') c + (\mu + \mu')d)/2,$$
$$d = a_+ - b_+ + a_0^- - b_0^- +(a_0^0 - b_0^0 + (\lambda - \lambda ') c + (\mu - \mu') d)/2,$$
$$e = b_- - a_- + b_0^+ - a_0^+ + (b_0^0 - a_0^0 + (\lambda' - \lambda) c + (\mu' - \mu) d)/2.$$
We have $c\oplus c\in \Hlie$ and $d\oplus e\in \mathfrak{K}$, and $a\oplus b = (c\oplus c) + (d\oplus e)$.

Let us show that $\Hlie$, $\mathfrak{K}$ are isotropic. Indeed for $a,b\in \hat{\Glie}$, we have $\langle a\oplus a, b\oplus b\rangle = (a,b) - (a,b) = 0$, and for $a = f + \lambda c + \mu d\oplus g -\lambda c -\mu d$, $b = f' + \lambda' c + \mu' d\oplus g' -\lambda' c - \mu' d \in \Hlie$, we have $\langle a,b\rangle = (f(0),f'(0)) - (g(\infty),g'(\infty)) +\lambda\mu' + \mu\lambda' - \lambda \mu' - \mu\lambda '  = 0$.

Let us show that $\Hlie$, $\mathfrak{K}$ are maximal isotropic. Let $a\oplus b\in\tilde{\Glie}$ such that $(a\oplus b, h\oplus h)=0$ for any $h\in \hat{gl_2}$. So $(a-b,h)=0$ for any $h\in \hat{\Glie}$. So $a = b$ and $a \oplus b\in \mathfrak{K}$. Let $a\oplus b\in\tilde{\Glie}$ such that $\langle a\oplus b,d\oplus e\rangle=0$ for any $d\oplus e\in\Hlie$. We can suppose that $a = b$. So $(a,d-e)=0$, that is to say $(a,h)=0$ for any $h\in \hat{\Glie}$ and $a=0$.\qed

As $\hat{\mathfrak{g}}$ is graded by finite dimensional vector spaces, we have an isomorphism $\mathfrak{K}\simeq
\Hlie^*$. In the following we identify the two spaces (see \cite[Section 6.1]{es} for Manin triples in the infinite dimensional case).

From the construction in Section \ref{ph}, we have a Poisson group structure on $H$. One of the main result of the present paper is a geometric realization of $Z_\epsilon$ in terms of the geometry of $H$ : we prove that $Z_\epsilon$ is isomorphic as a Hopf Poisson algebra to an algebra of functions on $H$  that is a ring of maps (which are polynomial in coordinates corresponding to the Drinfeld generators) which separate points. The proof is based on a reduction to rank $1$ and $2$, that is why we first study in more details $GL_1$ and $GL_2$, $SL_2$. From now on whenever the group $G$ and its Lie algebra have been specified we shall denote $LG$ simply by $L$.

\section{Quantum Heisenberg algebra and $GL_1$ analytic loop group}\label{gl1}

In this section we treat the case of $GL_1$. It is of particular importance as the quantum Heisenberg algebra appears in this situation, and with this toy example we see why a new notion of analytic loop group is necessary for our study and how the elliptic curve $\mathcal{E}$ comes into the picture. Moreover in this case most of the structures can be written very explicitly.

We study a
Riemann-Hilbert factorization (Theorem \ref{decompa}). The main result
(Theorem \ref{isom1}) identifies the Poisson structure with the semi-classical limit of the center of $\U_q(\hat{\mathfrak{gl}_1})$. The symplectic leaves are described in Proposition
\ref{sympl}.

\subsection{The analytic loop group}

In the $GL_1$-case the analytic loop group $L$ is commutative as $GL_1$ is commutative. The Riemann-Hilbert factorization is :

\begin{thm} An element $f\in L$ can be uniquely written in the form $f = f_+ z^{n(f)} f_-$ where $f_{\pm}\in L^{0,\pm}$, $n(f)\in\ZZ$, $f_-(\infty) = 1$.
\end{thm}

For $f = \text{exp}(g)$ and $g = g_+ + g_-\in \mathcal{L}$ decomposed as in Lemma \ref{dlie}, we have $f_{\pm} = \text{exp}(g_{\pm})$. $L^0$ is the Lie subgroup $f\in L$ such that $n(f)=0$ as by Theorem \ref{decompa} any $f\in L^0$ is of the form $f = e^g$ where $g\in\mathcal{L}$.

Although in general we will work with an extension  of $L^0$ as explained in the previous section, in the $GL_1$-case an extension $\hat{L}$ of $L$ can be written explicitly. The connected component of $1$ in $\hat{L}$ is the Lie subgroup $\hat{GL_1}$.

Consider the set $\hat{L} = L\times \CC^* \times \CC^*$. An element of $\hat{L}$ is a triple $(z^ne^f,\Lambda,\Gamma)$ where $f\in\mathcal{L}, \Lambda,\Gamma\in\CC^*, n\in\ZZ$. We define a product on $\hat{L}$ by :
$$(z^ne^f ,\Lambda,\Gamma)(z^m e^g,\Lambda',\Gamma')$$
$$= (z^{n+m}(\Gamma')^{-n}\Gamma^m e^{f(z(\Gamma')^{-1}) + g(z\Gamma)},
\Lambda \Lambda' \text{exp}((\frac{d}{dz}(f(z(\Gamma')^{-1}))g(z\Gamma))_{-1}),\Gamma\Gamma').$$
We have $(z^ne^f,\Lambda,\Gamma)^{-1} = (z^{-n}e^{-f}, \Lambda^{-1},\Gamma^{-1})$.
The choice of the formula, in particular of the term $(\frac{d}{dz}(f(z(\Gamma')^{-1})g(z\Gamma))_{-1}$, is done so that:
\begin{lem} The group structure on $\hat{L}$ is well-defined.
\end{lem}

\demo The only point to be checked is the associativity. For the first and third term this is straightforward. As for  the middle term, we have
$$(\frac{d}{dz}(f(z(\Gamma')^{-1})g(z\Gamma))_{-1}=(f'(z)g(z\Gamma\Gamma'))_{-1},$$
and so $((z^ne^f,\Lambda,\Gamma)(z^m e^g,\Lambda',\Gamma'))(z^pe^h,\Lambda'',\Gamma'')$ gives for the middle term:
$$[f'(z)g(z\Gamma\Gamma')]_{-1} + [\frac{d}{dz}(f(z(\Gamma')^{-1}) + g(z\Gamma)) h(z\Gamma\Gamma'\Gamma'')]_{-1}$$
$$=[f'(z)g(z\Gamma\Gamma')]_{-1} + [f'(z)h(z(\Gamma')^2\Gamma\Gamma'')]_{-1} + [g'(z)h(z\Gamma'\Gamma'')]_{-1}$$
and $(z^ne^f,\Lambda,\Gamma)((z^m e^g,\Lambda',\Gamma')(z^pe^h,\Lambda'',\Gamma'')$) gives  the same middle term:
$$[g'(z)h(z\Gamma'\Gamma'')]_{-1} + [f'(z)(g(z\Gamma\Gamma')+ h(z(\Gamma')^2\Gamma\Gamma''))]_{-1}.$$\qed

\begin{rem}\label{twform}  We have the following "twisted" commuting relation:
$$(z^ne^f,\Lambda,\Gamma)(z^m e^g,\Lambda',\Gamma')$$
$$= x (\Gamma^{2m}z^me^{g(z\Gamma^2)},\Lambda',\Gamma')((\Gamma')^{-2n}z^n e^{f(z(\Gamma')^{-2})},\Lambda,\Gamma).$$
where $x$ is the central element :
$$x = (1,\text{exp}(2(\frac{d}{dz}(f(z))g(z\Gamma\Gamma'))_{-1}),1).$$
\end{rem}

In $\hat{\mathfrak{gl}_1} = \mathcal{L} \oplus \CC c\oplus \CC d$, from Remark \ref{twform}, for $f,g\in\mathcal{L}$, $\lambda,\lambda',\mu,\mu'\in\CC$, the bracket is given by:
$$[f + \lambda c + \mu d,g + \lambda' c + \mu' d] = \text{Res}_{z=0}(gdf) c - \mu ' z \frac{df}{dz} + \mu z \frac{dg}{dz}.$$

We consider $\tilde{GL_1},H,K$. We have moreover the descriptions :
$$H = \{((f_+,\Lambda ,\Gamma),(f_-,\Lambda^{-1},\Gamma^{-1}))\in\tilde{GL_1}| f_\pm\in (L^0)^\pm, f_+(0) = f_-(\infty)^{-1}\},$$
$$\Hlie = \{(a + \lambda c + \mu d)\oplus (b - \lambda c - \mu d)\in \tilde{\mathfrak{gl}_1} | a\in \mathcal{L}^+, b\in \mathcal{L}^-,a(0) = - b(\infty) \}.$$

\subsection{Poisson bracket}\label{pb} The Manin triple constructed in the previous sections should induce a Poisson bracket of suitable rings of functions on the groups $H$ and $K$ and give them the structure of Poisson Lie groups. We are going to analyze the case of the group $H$.

 An element $h\in H$ can be written uniquely as
  $$h= ((ae^{\sum_{n >0}a_nz^n},\lambda ,\gamma),(a^{-1}e^{\sum_{n <0}a_nz^n}, \lambda^{-1}, \gamma^{-1}))\in H.$$ $a,\lambda, \gamma\in \mathbb C^*$, $a_n\in\mathbb C$ for $n\in\mathbb Z\setminus\{0\}$. We then define the functions  $k$, $\Lambda$, $\Gamma$, $h_m$, $m\in\mathbb Z\setminus\{0\}$  by
  $$k(h)=a,\ \ \Lambda(h)=\lambda, \ \ \Gamma(h)=\gamma, \ \ h_m(h)=a_m.$$

 We now take the ring $\CC[H]$ as the ring of functions on $H$ which are polynomials in the functions $h_m$, $k^{\pm 1},\ \Lambda^{\pm 1},\ \Gamma^{\pm 1}$. It is clear that as a ring $\CC[H]$ is just the polynomial ring $\CC[\Lambda,\Gamma,k, h_m]_{m\in\ZZ\setminus \{0\}}$ with
   $k,\ \Lambda,\ \Gamma$ inverted.

   Thus by mapping each of the generators of $\CC[H]$ to the corresponding generator of the reduction $Z_1$ of $\mathcal{U}_{q}(\hat{\mathfrak{gl}}_1)$ at $q=1$ considered in section 3, we get an obvious isomorphism $\Phi:\CC[H]\to Z_1$.

   Using the definition of the product on  $H$ we immediately get the following Lemma whose  proof  we leave to the reader.
    \begin{lem}\label{lahop} Let $h, h'\in H$. Then
    $$k(hh')=k(h)k(h'),\ \ \Lambda(hh')=\Lambda(h)\Lambda(h'),\ \ \Gamma(hh')=\Gamma(h)\Gamma(h'),$$ $$ h_m(hh')=h_m(h)\Gamma^{-|m|}(h')+\Gamma^{|m|}(h)h_m(h'),$$
    $$k(h^{-1})=k(h)^{-1},\ \ \Lambda(h^{-1})=\Lambda(h)^{-1},\ \ \Gamma(h^{-1})=\Gamma(h)^{-1},\ \ h_m(h^{-1}) = - h_m(h),$$
    $$k(1) = 1,\ \ \Lambda(1)=1,\ \ \Gamma(1)=1,\ \ h_m(1) = 0,$$
    for each $m\in\mathbb Z\setminus \{0\}$.
    In particular we get a Hopf algebra structure on $\CC[H]$ with respect to which $\Phi$ is an isomorphism of Hopf algebras.
    \end{lem}

Our second task is the computation of the Poisson bracket of two elements in $\CC[H]$ with respect to the Poisson structure induced by the Manin triple.

    The main result of this Section is

\begin{thm}\label{isom1} The Poisson bracket of two elements in $\CC[H]$ lies in $\CC[H]$. Furthermore
$\Phi:\CC[H]\to Z_1$ is an isomorphism of    Poisson Hopf algebras.
\end{thm}

\demo In the following, to simplify notations we are going to write, for $h\in H$,  just $\Gamma , \Lambda , h_m$ for $\Gamma(h),\Lambda(h),h_m(h)$.

Fix  $h, h'\in H$. By Lemma \ref{lahop} we have for $m > 0$:
$$h_m(hh') = h_m(h)(\Gamma(h'))^{-m} + (\Gamma(h))^m h_m(h'),$$
 We deduce:
\begin{equation*}
\begin{split}
dl_h^*(dh_m) &= \Gamma^m dh_m - h_m m \Gamma^{-1}d\Gamma\in \Hlie^*
\\&= (\Gamma^m z^{-m} - h_m m c) \oplus (\Gamma^m z^{-m} - h_m m c)\in\mathfrak{K}.
\end{split}
\end{equation*}
In the same way for $m < 0$:
$$dl_h^*(d h_m) = -((\Gamma^{-m}z^{-m} + h_m m c) \oplus(\Gamma^{-m}z^{-m} + h_m m c)).$$
So for $m > 0$ we get that $Ad(h)(dl_h^*(dh_m))$ is equal to :
\begin{equation*}
\begin{split}
 & \text{exp}(h_+)z^{-m}\Gamma^{-m}\text{exp}(-h_+) - h_m m c \oplus \text{exp}(h_-)\Gamma^{3m} z^{-m}\text{exp}(-h_-) - h_m m c
\\=& z^{-m} \Gamma^{-m}\oplus  \Gamma^{3m} z^{-m} - h_m m c.
\end{split}
\end{equation*}
This implies :
$$\pi(Ad(h)(dl_h^*dh_m)) = (0\oplus (\Gamma^{3m} - \Gamma^{-m})z^{-m}) + mh_m (c\oplus (-c))/2,$$
$$\pi_h(dl_h^*dh_m) = (0\oplus (\Gamma^m - \Gamma^{-3m})z^{-m}) + mh_m (c\oplus  (-c))/2.$$
So for $m' > 0$ we have $\{h_m, h_{m'}\} = 0$ and for $m' < 0$ :
$$\{h_m , h_{m'}\}  = \Gamma^{-m'}(\Gamma^{m} - \Gamma^{-3m}) (z^{-m}, z^{ -m'}) = (\Gamma^{2m} - \Gamma^{-2m}) \delta_{m,-m'}.$$
In the same way $\{h_m,h_{m'}\} = 0$ for $m,m' < 0$, and for $m,m'\in\ZZ-\{0\}$ :
$$\{h_m , h_{m'}\} = (\Gamma^{2m} - \Gamma^{-2m}) \delta_{m,-m'}.$$
We have for $h,h'\in H$
$$\Lambda(hh') = \Lambda(h) \Lambda(h')\text{ , }\Gamma(hh') =\Gamma(h)\Gamma(h')\text{ , }k(hh') = k(h)k(h').$$
So the images by $dl_h^*$ respectively of $d\Gamma$, $d\Lambda$, $dk$ are :
$$d\Gamma = \Gamma (c\oplus (-c))\text{ , }dl_h^*(d\Lambda) = d\Lambda = \Lambda(d\oplus (-d))\text{ , }
dl_h^*(dk) = k(1\oplus 1)/2.$$
So $\{\Gamma,\CC[H]\} = \{k,\CC[H]\} = 0$, and for $m\neq 0$ :
$$\{h_m, \Lambda\} = m\Lambda h_m.$$
This proves all our claims.
\qed

\subsection{Symplectic leaves} In this section we describe the symplectic leaves of $H$ (Proposition \ref{sympl}). We also see how an elliptic curve comes into the picture.

\subsubsection{Description}

Consider the map
$$\gamma : \tilde{GL_1}\rightarrow \hat{GL_1}\text{ , }\gamma(a,b) = a^{-1}b.$$
We have  remarked that  $H\cap K=(\mathbb Z/2\mathbb Z)^3$ and we claim that the restriction of $\gamma$ to $H$ is a Galois covering with group  $H\cap K$. Indeed for $f^\pm\in\mathcal{L}^{\pm}$ without constant term and $\alpha,\Lambda,\Gamma\in\CC^*$, we have
$$\gamma^{-1}((\alpha^2 e^{f_+ + f_-},\Lambda^2,\Gamma^2)) $$
$$= \{((\epsilon_1\alpha^{-1}e^{-f_+},\epsilon_2\Lambda^{-1},\epsilon_3\Gamma^{-1}),(\epsilon_1\alpha e^{f_-},\epsilon_2\Lambda,\epsilon_3\Gamma)),\epsilon_1,\epsilon_2,\epsilon_3\in \mathbb Z/2\mathbb Z\}.$$
Thus two elements in the fiber differ by multiplication times  $((\epsilon_1,\epsilon_2,\epsilon_3),(\epsilon_1,\epsilon_2,\epsilon_3))\in H\cap K$.

We have a commutative diagram :
$$\begin{CD}H@>{i}>>\tilde{GL_1}\\ @V{\gamma}VV  @V{p}VV\\\hat{GL_1}@<{j}<<  K\setminus \tilde{GL_1}\end{CD}  ,$$
 where $i$ is the inclusion, $p$ is the projection and $j$ is the obvious isomorphism. By Proposition \ref{symp}, the symplectic leaves of $H$ are the connected components of the  preimages under $p\circ i$ of $K$ orbits in $K\setminus \tilde{GL_1}$, thus they are the connected components of  preimages under $\gamma$ of conjugacy classes in $\hat{GL_1}$.  $\Gamma$  is preserved along the symplectic leaves of $H$ since it could only change his sign.   So fix the value of $\Gamma$ and assume that $\Gamma$  is not a root of unity.

 We have
 \begin{lem} Let $h=(e^f,\lambda,\Gamma)\in \hat {GL_1}$. Assume that $\Gamma$ is not a root of unity.  Then $h$ is conjugated to a unique element  of the form
 $(\alpha,\lambda',\Gamma)$ with $\alpha\in \mathbb C$.
 \end{lem}
 \begin{proof} Write $f=\sum_{n\in \mathbb Z}a_nz^n$. If we take an element $(e^g,\mu,\rho)$ with $g=\sum_{n\in \mathbb Z}c_nz^n$ and we conjugate by it, we get that $h$ is transformed into the element
 $$(e^{g(z\rho^{-1}\Gamma)+f(z\Gamma^2)-g(z\rho\Gamma)},\tilde \lambda,\rho)$$
 where $\tilde \lambda$ is a suitable explicit non zero complex number.
  If we expand $g(z\rho\Gamma^{-1})+f(z\rho^2)-g(z\rho\Gamma)=\sum_{n\in \mathbb Z} b_nz^n$ we get that for every integer $n$
  $$b_n=c_n\rho^n(\Gamma^{-n}-\Gamma^n)+a_n\rho^{2n}.$$
  If for $n\neq 0$ we take
  $$c_n=\frac {a_n\rho^n}{\Gamma^{n}-\Gamma^{-n}}$$
  which we can do since $\Gamma$ is not a root of unity. We have conjugated $h$ to an element of the form $(\alpha,\lambda',\Gamma)$.

  It remains to see that these elements are pairwise non conjugated. In fact for $(\alpha,\lambda,\Gamma)$ conjugated by $(e^g,\mu,\rho)$ to another such element, $g$ is constant by the above computation, and the resulting element is $(\alpha,\lambda,\Gamma)$.\end{proof}

  Let us now assume we are in the generic case  $|\Gamma|\neq 1$. In the next Proposition, we compute explicitly the middle term, and so we get an explicit description of the symplectic leaves :

\begin{prop}\label{sympl} The symplectic leaves of $H$ are the $(H_{\alpha,\Lambda})_{\alpha,\Lambda\in\CC^*}$ where
$$H_{\alpha,\Lambda} = ((\alpha^{-1},\Lambda^{-1},1),(\alpha,\Lambda,0))H_{1,1}$$
and $H_{1,1}$ is equal to the set of elements of the form :
$$((e^{f_+(t)},\text{exp}(- \sum_{n\neq 0}\frac{n(f)_n(f)_{-n}}{2(1-\Gamma^{-4n})^2}) , \Gamma^{-1}),(e^{f_-(t)},\text{exp}(\sum_{n\neq 0}\frac{n(f)_n(f)_{-n}}{2(1-\Gamma^{-4n})^2}),\Gamma))\in H$$
where $f_\pm\in \mathcal{L}^\pm$, $f_-(\infty) = f_+(0) = 0$, and $f = f_- - f_+\in\mathcal{L}$.
\end{prop}

\demo
 First remarks that in the formulas of the lemma the middle term is well defined. Indeed let $R > 0$ such that $f_+$ is holomorphic on $\mathcal{D}_R$. We know that $f_-$ is holomorphic on $\mathbb{P}_1(\CC)^*$. So the term equal to
$$-\sum_{n > 0}\frac{n(f_+)_nR^n}{(\Gamma^{2n} - \Gamma^{-2n})^2}((f_-)_{-n}R^{-n}\Theta^{n}) - \sum_{n < 0}\frac{n(f_+)_{-n}R^{-n}}{(\Gamma^{2n} - \Gamma^{-2n})^2}((f_-)_nR^n\Theta^{n}).$$
is a convergent sum.

Let $\alpha,\Lambda\in\CC^*$ and consider the following element
$$h = ((\alpha^{-1},\Lambda^{-1},\Gamma^{-1}),(\alpha,\Lambda,\Gamma))\in H.$$
Let $\tilde{g} = (e^{g(t)},\Lambda',\Gamma')\in \hat{GL_1}$ where $\Lambda',\Gamma'\in\CC^*$ and $g\in\mathcal{L}$. Then we have
$$\tilde{g}^{-1}\gamma(h) = (\alpha^2e^{-g(t\Gamma^{-2})},\Lambda^2(\Lambda')^{-1},\Gamma^2(\Gamma')^{-1}),$$
$$\tilde{g}^{-1}\gamma(h)\tilde{g} = (\alpha^2e^{F(t)},\Lambda'',\Gamma^2)$$
where $F(t) = g(t \Gamma^2 (\Gamma')^{-1})-g(t \Gamma^{-2}(\Gamma')^{-1})$ and
$$\Lambda'' = \Lambda^2 \text{exp}(-[\frac{d}{dt}(g(t\Gamma^{-2}(\Gamma')^{-1}))g(t \Gamma^2 (\Gamma')^{-1})]_{-1}).$$
We get the result as $(F)_n = (g)_n(\Gamma')^{-n}(\Gamma^{2n}-\Gamma^{-2n})$.
\qed

\subsubsection{Elliptic curve}

The symplectic leaves are parameterized by $\CC^*\times \CC^*$. In fact if we consider the whole loop group $L$ instead of $L^0$, we have additional elements $z^n$ for $n\neq 0$. The conjugation by these elements changes $\alpha$ to elements in $\alpha \Theta^{\ZZ}$ and we get a parametrization by $\mathcal{E}\times \CC^*$ where $\mathcal{E}$ is the elliptic curve defined in Section \ref{manin}. Such parameterizations related to $\mathcal{E}$ will appear as well for other types.

\section{Analytic loop group for $GL_2$ and $SL_2$}\label{sldeux}

In this section we treat the case of $GL_2$. This case is crucial as we will see in the last section that the general result can be proved by using a reduction to this case. We study a Riemann-Hilbert factorization of the analytic loop group (Theorem \ref{decompb}). In particular we see how the relation between Drinfeld generators
appear and the main result is Theorem \ref{isomt} where the Poisson structure is identified with the semi-classical limit of the center of the quantum affine algebra. The holomorphic symplectic leaves are parameterized in Theorem \ref{sl}.

\subsection{Analytic loop group}
$SL_2\subset GL_2$ is simply-connected, $L SL_2\subset L GL_2$ is connected, but $LGL_2$ is not.

Remark: in addition to the good properties  that we have already discussed, $L GL_2$ includes elements  of the form $A + z^{\pm 1}B$ with $\text{det}(A)\neq 0$. So  $LSL_2$ contains elements of the form $(A + z^{\pm 1}B)(\text{det}(A + z^{\pm 1}B))^{-\frac{1}{2}}$.

We can consider the Fourier coefficients $\mathcal{F} = \{\mathcal{F}_n\}_{n\in\ZZ}$ of elements in $LGL_2$. In the following $\mathcal{F}_n(f)$ will be denoted by $(f)_n$.

In this case the Riemann-Hilbert  factorization takes the form:
\begin{thm}\label{decompb} An element $f\in LGL_2$ can be written in the form
$$f = f_+ \text{diag}(z^{n_1},z^{n_2}) f_-$$
 where $f_{\pm}\in {LGL_2}^\pm$, $n_1,n_2\in\ZZ$.
\end{thm}

Remark : $n(f) = n_1 + n_2$ is uniquely determined by $f$ as it is equal to $n(\text{det}(f))$ in the sense of the section on $GL_1$. It is the winding number (as defined for example in \cite[Section 13]{hsw}). Moreover $n_1(f) = n_1$ and $n_2(f) = n_2$ are also uniquely determined by $f$ (Theorem \ref{birk}). As for the case of usual loops, we have :

\begin{lem} The connected component of $1$ in $L GL_2$ is
$$LGL_2^0 = \{f\in L GL_2|n(f) = 0\}.$$\end{lem}

\demo The connected component of $1$ is included in $\{f\in L GL_2|n(f) = 0\}$. Then for $f\in L GL_2$, we have from Theorem \ref{decompb}
$$f = \frac{f_+}{(\text{det}(f_+))^{\frac{1}{2}}}(\text{diag}(z^{n_1},z^{-n_1})(\text{det}(f_+)\text{det}(f_-))^{\frac{1}{2}})  \frac{f_-}{(\text{det}(f_-)^{\frac{1}{2}})}.$$
Note that the root square can be chosen so that it is well-defined on the connected domain where the respective $\text{det}$ are non zero. As $L SL_2$ is connected, it suffices to prove that $(\text{det}(f_+)\text{det}(f_-))^{\frac{1}{2}} I_2$ is in the connected component of the identity. This follows from the $GL_1$-case studied above.
\qed

\subsection{Poisson bracket} In order to prove Theorem \ref{isomt}, we are now going to compute more explicitly the bracket defined in Section \ref{manin}.

We use the strategy of \cite{dcp}: we compute left invariant forms. We treat in this section the case $SL_2$, the results are analog for $GL_2$ (with modifications explained at the end of the Section).

An element $M\in SL_2$ can be written in the form $M = e^{M_+}e^{M_0}e^{M_-}$ where $M_\pm\in \mathfrak{n}^\pm$, $M_0\in\mathfrak{d}$ if and only if $M_{1,1}\neq 0$ (this is the big cell of $SL_2$). In the same way we can consider the big cell of $H$, that is to say of elements $h\in H$ written in the form :
$$h = ((\text{exp}(h_+^+)\text{exp}(h_+^0)\text{exp}(h_+^-),\Lambda,\Gamma),(\text{exp}(h_-^+)\text{exp}(h_-^0)\text{exp}(h_-^-),\Lambda^{-1},\Gamma^{-1}))$$
where $h_\pm ^+$ (resp. $h_\pm^-$, $h_\pm^0$) have value in $\mathfrak{n}^+$ (resp. $\mathfrak{n}^-$, $\mathfrak{d}$).

$\Lambda$, $\Gamma$, $\text{exp}(h_+)(0) = (\text{exp}(h_-)(\infty))^{-1} = k$ depend only of $h$ and so give maps $\Lambda,\Gamma : H\rightarrow \CC^*$ and $k : H\rightarrow \CC^*$ (for $\Lambda$ see Remark \ref{wdef}). For $r\in\ZZ\setminus\{0\}$, $h_r = 2((h_+^0)_{1,1} - (h_-^0)_{1,1})_r$ depends only of $h$ (this is analog to the case $GL_1$) and for $m\in\ZZ$, $x_m^+ = ((h_+^+ + h_+^-)_{1,2})_m$, $x_m^- = ((h_-^+ + h_-^-)_{2,1})_m$ depend only of $h$ (because they are Fourier coefficients of coefficients of $\text{exp}(h_{\pm}^+), \text{exp}(h_{\pm}^-)$).

We work with the coordinate ring $\CC[H] = \CC[\Lambda^{\pm 1},\Gamma^{\pm 1},k^{\pm 1}, x_m^\pm,h_r]_{m\in\ZZ,r\in\ZZ-\{0\}}$ of maps which are polynomial in the $\Lambda^{\pm 1}$, $\Gamma^{\pm 1}$, $k^{\pm 1}$, $x_m^\pm$, $h_r$. The main result of this Section is :

\begin{thm}\label{isomt} The Poisson bracket of two elements in $\CC[H]$ lies in $\CC[H]$. Furthermore
$\CC[H]$ and $Z_\epsilon$ are isomorphic as Poisson Hopf algebras.
\end{thm}

First note that for two elements in $\hat{SL_2}$ of the form :
$$A \in  (e^{aX_-}e^{bH}e^{eX_+},\CC^*,\Gamma)\text{ , } B = (e^{a'X_-}e^{b'H}e^{e'X_+},\CC^*,\Gamma'),$$
we have formally :
$$AB \in (\text{exp}(X_-(a(t) + \frac{e^{-2b(t\Gamma)} a'(t\Gamma^2)}{1+e(t\Gamma^2) a'(t\Gamma^2)}))$$
$$\times\text{exp}(H_-(b(t(\Gamma')^{-1})+b'(t\Gamma) + ln(1+a'(t\Gamma (\Gamma')^{-1})e(t\Gamma (\Gamma')^{-1}))))$$
$$\times\text{exp}(X_-(e'(t) + \frac{e^{-2b'(t(\Gamma')^{-1})}e(t(\Gamma')^{-2})}{1+e(t(\Gamma')^{-2})a'(t(\Gamma')^{-2})})),\CC^*,\Gamma\Gamma'),$$
$$A^{-1} \in (\text{exp}(X_-\frac{- a(t\Gamma^{-2})e^{2b(t\Gamma^{-1})}}{1 + e(t)e^{2b(t\Gamma^{-1})}a(t\Gamma^{-2})})$$
$$\times\text{exp}(H(-b(t) + ln(1+e(t\Gamma)e^{2b(t)}a(t\Gamma^{-1}))))
\times\text{exp}(X_+\frac{-e(t\Gamma^2)e^{2b(t\Gamma)}}{1+e(t\Gamma^2)a(t)e^{2b(t\Gamma)}}),\CC^*,\Gamma^{-1}).$$

Here we use an abuse of notation by writing $\CC^*$ for the central extension (see Remark \ref{wdef}). But this is not a problem as we do not compute anything for the central elements here.

We define $x^{+,\pm}(z), h^+(z)\in \CC[H][[z]]$ and $x^{-,\pm}(z), h^-(z)\in \CC[H][[z^{-1}]]$ :
$$x^{\pm ,\pm}(z) = \sum_{m > 0} x_{\mp m}^{\pm} z^{\pm m} \in z^{\pm 1}\CC[[z^{\pm 1}]],$$
$$x^{\pm ,\mp}(z) = \sum_{m\geq 0} x_{\mp m}^{\mp} z^{\pm m} \in \CC[[z^{\pm 1}]],$$
$$h^{\pm }(z) = \sum_{m > 0} h_{\mp m} z^{\pm m} \in z^{\pm 1}\CC[[z^{\pm 1}]].$$
We will also consider $\phi^{\pm}(z) = k^{\pm 1}e^{\pm h^{\pm}(z)}$.

The structures of $\CC[H]$ are extended to $\CC[H][[z^{\pm 1}]]$ in an obvious way.

\begin{lem} We have :
$$\Delta(x^{\pm ,-}(z)) = x^{\pm ,-}(z) \otimes 1 + \frac{(\phi^{\pm})^{-1}(z\Gamma^{\pm 1})\otimes x^{\pm ,-}(z\Gamma^{\pm 2}\otimes 1)}{1 + x^{\pm,+}(z\Gamma^{\pm 2})\otimes x^{\pm ,-}(z\Gamma^{\pm 2}\otimes 1)},$$
$$\Delta(x^{\pm,+}(z)) = 1 \otimes x^{\pm,+}(z) + \frac{x^{\pm,+}(z\otimes \Gamma^{\mp 2})\otimes (\phi^{\pm})^{-1}(z\Gamma^{\mp 1})}{1 + x^{\pm,+}(z\otimes \Gamma^{\mp 2})\otimes x^{\pm,-}(z\Gamma^{\mp 2})},$$
$$\Delta(h^{\pm}(z)) = h^{\pm}(z\otimes \Gamma^{\mp 1})\otimes 1 + 1\otimes h^{\pm}(z \Gamma^{\pm 1} \otimes 1) $$
$$\pm 2ln(1 + x^{\pm,+}(z\Gamma^{\pm 1}\otimes \Gamma^{\mp 1})\otimes x^{\pm,-}(z \Gamma^{\pm 1}\otimes\Gamma^{\mp 1})),$$
$$\Delta(k) = k\otimes k\text{ , }\Delta(\Gamma) = \Gamma \otimes \Gamma',$$
$$S(x^{\pm,-}(z)) = - \frac{\phi^{\pm}(z \Gamma^{\mp 1})x^{\pm,-}(t\Gamma^{\mp 2})}{ 1 + x^{\pm,+}(z)x^{\pm,-}(z\Gamma^{\mp 2})\phi^{\pm}(z\Gamma^{\mp 1})},$$
$$S(x^{\pm,+}(z)) = - \frac{\phi^{\pm}(z\Gamma^{\pm 1})x^{\pm,+}(z\Gamma ^{\pm 2\mu})}{ 1 + x^{\pm,+}(z\Gamma^{\pm 2})x^{\pm,-}(z)\phi^{\pm}(z\Gamma^{\pm 1})},$$
$$S(h^{\pm}(z))     = - h^{\pm}(z)  \pm 2ln(1 + x^{\pm,+}(z\Gamma^{\pm 1})x^{\pm,-}(z\Gamma^{\mp 1})\phi^{\pm}(z)),$$
$$S(k) = k^{-1}\text{ , }S(\Gamma) = \Gamma^{-1}.$$
\end{lem}
Note that all formulas involved in the Lemma make sense in $\CC[H][[z^{\pm 1}]]$.

The left invariant differential form coinciding with $dk$ in $1$ is $k^{-1}dk$. We have :

\begin{lem}\label{linv} The left invariant differential forms respectively coinciding in $1$ with $d(x^{+,-}(z)),d(x^{+,+}(z)),  d(h^+(z))$ are equal modulo $\CC[H][[z]] \Gamma^{-1}d\Gamma$ to :
$$d_{+,-}(z) = \phi_+(z\Gamma^{-1}) d(x^{+,-})(z\Gamma^{-2}),$$
$$d_{+,+}(z) = d(x^{+,+})(z) + x^{+,+}(z) d(h^+)(z\Gamma^{-1}) $$
$$- \phi_+(z\Gamma^{-1}) (x^{+,+}(z))^2 d(x^{+,-})(z\Gamma^{-2}),$$
$$d_+(z) = d(h^+(z\Gamma^{-1})) - 2\phi_+(z\Gamma^{-1})x^{+,+}(z) d(x^{+,-}(z\Gamma^{-2})).$$
\end{lem}

\demo By using Lemma \ref{inv}, we get that $d_{+,-}(z\Gamma^{-1})$ is equal modulo $\CC[H][[z]] \Gamma^{-1}d\Gamma$ to :
$$\frac{S(\phi^+)(z\Gamma^{-1})(1 + S(x^{+,+})(z\Gamma^{-2}) x^{+,-}(z\Gamma^{-2})) d(x^{+,-})(z\Gamma^{-2})}{(1 + S(x^{+,+})(z\Gamma^{-2}) x^{+,-}(z\Gamma^{-2}))^2}$$
$$- \frac{S(\phi^+)(z\Gamma^{-1}) x^{+,-}(z\Gamma^{-2})S(x^{+,+})(z\Gamma^{-2}) d(x^{+,-})(z\Gamma^{-2})}{(1 + S(x^{+,+})(z\Gamma^{-2}) x^{+,-}(z\Gamma^{-2}))^2}.$$
This gives the result. The proof is analog for the other forms.\qed

We now use these forms to compute brackets and to prove Theorem \ref{isomt}. Let us give some technical results on series with two parameters which will be useful in the following. For $a$ a smooth map $S_1 \rightarrow \CC$ and $m\in\ZZ$, we denote by $(a)_{\geq m}$ the truncation of $a$  at $m$, that is the smooth map such that $ ((a)_{\geq m})_p=(a)_p $ if $p\geq m$ and $((a)_{\geq m})_p = 0$ otherwise. We use an analogous definition for $(a)_{>m},(a)_{\leq m}, (a)_{< m}$. For a formal variable $X$, $a(X)$ denotes the formal sum $\sum_{r\in \ZZ}a_r X^r$.

\begin{lem}\label{aideform} Let $a,e$ be smooth maps $S_1 \rightarrow \CC$ such that $(a)_m = (e)_m = 0$ for $m < 0$. Then :
$$\sum_{m,m'\geq 0}w^{m'}z^m (e(a)_{\geq m})_{m+m'} = e(w)\sum_{m,m'\geq 0} w^{m'}z^{m} (a)_{m + m'},$$
$$\sum_{m,m'\geq 0}w^{m'}z^m (e(a)_{\leq m})_{m + m'} = a(z)\sum_{m,m'\geq 0} w^{m'}z^m (e)_{m + m'},$$
$$\sum_{m\geq 0,m' > 0}w^{m'}z^m (e(a)_{\leq m})_{m + m'} = a(z)\sum_{m\geq 0,m' > 0} w^{m'}z^m(e)_{m + m'}.$$
\end{lem}

\demo For the first equality, the left term is equal to:
\\$\sum_{m,m'\geq 0,r\geq 0}w^{m'}z^m e_{ m' - r} a_{m + r})
\\= \sum_{m'\geq 0}\sum_{r\geq 0}e_{m' - r}w^{m'-r} (\sum_{m\geq 0} w^rz^m a_{m + r}),$
which is equal to the right term. For the second equality the left term is equal to :
\\$\sum_{m,m'\geq 0, r\leq m}w^{m'}z^m e_{m + m' -r}a_r
\\= \sum_{m\geq 0}\sum_{r\leq m} z^ra_r (\sum_{m'\geq 0}w^{m'}z^{m-r} e_{m' + m -r}),$
which is equal to the right term. The last equality is proved in the same way.\qed

We use similar notations and results for germs of analytic maps on the punctured disc. Now let us prove Theorem \ref{isomt}:

\demo For the Poisson algebra isomorphism, this is a consequence of Lemma \ref{linv} : as we have left invariant forms it suffices to compute the brackets at $0$. Let us write in details the computation for the following formulas which hold in $Z_\epsilon \simeq Z_1$ (other brackets are computed in the same way) :
\begin{equation}\label{un}\{x_{-m}^-,x_{-m'}^-\} = 2(\sum_{r\leq m'}x_{-r}^- x_{-m-m'+r}^- - \sum_{r\leq m}x_{-r}^-x_{-m-m'+r}^-)\text{ for $m,m'\geq 0$.}\end{equation}
\begin{equation}\label{deux}\{x_{-m}^-,h_{-m'}\} = - 4\Gamma^{m'}x_{-m-m'}^-\text{ for $m\geq 0, m' > 0$.}\end{equation}
\begin{equation}\label{trois}\{x^-_{-m},x^+_{-m'}\} = -2 \Gamma^{m-m'} ((\phi^+(t))^{-1})_{m + m'}\text{ for $m\geq 0, m'>0$.}\end{equation}
\begin{equation}\label{quatre}\{k,x_{-m}^-\} = 2kx_{-m}^-.\end{equation}

For any $h \in H$ we have :
$$\{d_{+,-}(z),d_+(w)\}(h) = (\pi_h((d x_{+,-})_0(z)) , ((dh)_0(w))).$$
As $(d x_{+,-})_0(z) = 2\sum_{m\geq 0} (zt^{-1})^m (X_+\oplus X_+)$, the first member of $\pi_h((d x_{+,-})_0(z))$ is :
$$2\sum_{m\geq 0} (h_+)^{-1}(zt^{-1}\Gamma^{-2})^m \begin{pmatrix} - (x^{+,-}(t)\phi^+(t\Gamma))_{\geq m}& (\phi^+(t\Gamma))_{> m} \\ - ((x^{+,-}(t))^2\phi^+(t\Gamma))_{\geq m}& (x^{+,-}(t)\phi^+(t\Gamma)))_{\geq m}\end{pmatrix}h_+ $$
$$+ \sum_{m\geq 0} (h_+)^{-1}(zt^{-1}\Gamma^{-2})^m (x^{+,-}(t)\phi^+(t\Gamma))_m H h_+.$$
By computing $\{d_{+,-}(z),d_{+,-}(w)\}$ we get :

$\phi^+(z\Gamma^{-1})\phi^+(w\Gamma^{-1}) \{d(x^{+,-})(z\Gamma^{-2}),d(x^{+,-})(w\Gamma^{-2})\}
\\= 2\sum_{m,m' \geq 0}w^{m'}z^m ( \phi^+(t\Gamma^{-1}) (2 x^{+,-}(t\Gamma^{-2}) (x^{+,-}(t\Gamma^{-2})\phi^+(t\Gamma^{-1}))_{\geq m}
\\- (x^{+,-}(t\Gamma^{-2})^2\phi^+(t\Gamma^{-1}))_{\geq  m} - x^{+,-}(t\Gamma^{-2})^2(\phi^+(t\Gamma^{-1}))_{> m}))_{m + m'}
\\ - 2\sum_{m,m'\geq 0} w^{m'}(zt^{-1})^m ((x^{+,-}(t\Gamma^{-2})\phi^+(t\Gamma^{-1}))_m \phi^+(t\Gamma^{-1})x^{+,-}(t\Gamma^{-2}))_{m + m'}
\\=2\sum_{m,m' \geq 0}w^{m'}z^m ( \phi^+(t\Gamma^{-1}) (- (x^{+,-}(t\Gamma^{-2}) (x^{+,-}(t\Gamma^{-2})\phi^+(t\Gamma^{-1}))_{< m})_{\geq m}
\\+ x^{+,-}(t\Gamma^{-2}) (x^{+,-}(t\Gamma^{-2})(\phi^+(t\Gamma^{-1}))_{\leq m})_{\geq m}))_{m + m'}
\\-2\phi^+(z\Gamma^{-1})\phi^+(w\Gamma^{-1})x^{+,-}(w\Gamma^{-2})x^{+,-}(z\Gamma^{-2})
\\=-2\phi^+(w\Gamma^{-1})x^{+,-}(z\Gamma^{-2})\phi^+(z\Gamma^{-1})\sum_{m,m' \geq 0}w^{m'}z^m (x^{+,-}(t\Gamma^{-2}))_{m + m'}
\\+ 2\phi^+(w\Gamma^{-1}) x^{+,-}(w\Gamma^{-2})\phi^+(z\Gamma^{-1}) \sum_{m,m'\geq 0}w^{m'}z^m (x^{+,-}(t\Gamma^{-2}))_{m + m'}
\\+2\sum_{m,m' \geq 0}w^{m'}z^m (\phi^+(t\Gamma^{-1})(x^{+,-}(t\Gamma^{-2}) (x^{+,-}(t\Gamma^{-2})\phi^+(t\Gamma^{-1}))_m))_{m + m'} \\-2\phi^+(z\Gamma^{-1})\phi^+(w\Gamma^{-1})x^{+,-}(w\Gamma^{-2})x^{+,-}(z\Gamma^{-2})
.$

As the last two terms cancel, we get the following expression equivalent to Formula (\ref{un}):
$$\{x^{+,-}(z),x^{+,-}(w)\}
=2(x^{+,-}(w) -x^{+,-}(z))\sum_{m,m'\geq 0}w^{m'}z^m x^-_{-m-m'}.
$$

By computing $\{d_{+,-}(z),d_+(w)\}$, we also get :

$\phi^+(z\Gamma^{-1}) \{d(x^{+,-})(z\Gamma^{-2}),d(h^+(w\Gamma^{-1}))\} - 2x^{+,+}(w) \{d_{+,-}(z),d_{+,-}(w)\}
\\= 4\sum_{m\geq 0,m' > 0}w^{m'} z^m  (-(x^{+,-}(t\Gamma^{-2})\phi^+(t\Gamma^{-1}))_{> m}
\\+ 2x^{+,-}(t\Gamma^{-2})(\phi^+(t\Gamma^{-1}))_{> m})_{m + m'} - 2x^{+,+}(w)\{d_{+,-}(z),d_{+,-}(w)\}(h)$

and

$\phi^+(z\Gamma) \{d(x^{+,-})(z),d(h^+(w))\}
\\= 4\sum_{m\geq 0,m' > 0}(\Gamma w)^{m'} (z\Gamma^2)^m ( - x^{+,-}(t\Gamma^{-2})(\phi^+(t\Gamma^{-1}))_{\leq m})_{m + m'}
\\= -4\phi^+(z\Gamma)\sum_{m\geq 0,m' > 0}(w\Gamma)^{m'} z^m x_{m+m'}^-,$
\\and so the following expression equivalent to Formula (\ref{deux}) :
$$\{x^{+,-}(z),h^+(w)\} = -4\sum_{m\geq 0,m' > 0}w^{m'} z^m x_{-m-m'}^-\Gamma^{m'}.$$

Then by computing $\{d_{+,-}(z),d_{+,+}(w)\}$ we get :

$\phi^+(z\Gamma^{-1}) \{d(x^{+,-})(z\Gamma^{-2}) ,d(x^{+,+})(w) + x^{+,+}(w) d(h^+)(w\Gamma^{-1})
\\- \phi^+(w\Gamma^{-1}) (x^{+,+}(w))^2 d(x^{+,-})(w\Gamma^{-2})\}
\\
= 2\sum_{m\geq 0,m' > 0}w^{m'}z^m ((\phi^+)^{-1}(t\Gamma^{-1})(\phi^+(t\Gamma^{-1}))_{ > m})_{m + m'}
\\+ x^{+,+}(w) (\{d_{+,-}(z),d_+(w)\}(h)
+ x^{+,+}(w)^2\{d_{+,-}(z),d_{+,-}(w)\}(h))
$

and

$\phi^+(z\Gamma) \{x^{+,-}(z) , x^{+,+}(w)\}
\\=-2\sum_{m\geq 0,m' > 0}w^{m'}(\Gamma^2 z)^m ((\phi^+)^{-1}(t\Gamma^{-1})(\phi^+(t\Gamma^{-1}))_{\leq m})_{m + m'}$
\\and so the following expression equivalent to Formula (\ref{trois}) :
$$\{x^{+,-}(z) ,x^{+,+}(w)\}
= - 2\sum_{m\geq 0,m' > 0}w^{m'}(\Gamma^2 z)^m (\phi^+(t\Gamma^{-1}))_{m + m'}.$$

By computing $\{d_{+,-}(z),dk\}$ we get :
$$k^{-1}\phi^+(z\Gamma^{-1})\{d(x^{+,-}(z\Gamma^{-2})),dk\} = -2\sum_{m\geq 0}z^{-m}\Gamma^{-m}[\phi^+(t\Gamma)x^{+,-}(t)]_m$$
$$ = -2\phi^+(z\Gamma^{-1})x^{+,-}(z\Gamma^{-2})$$
and so Formula (\ref{quatre}) follows.

We have a Poisson algebra isomorphism. Now we prove the compatibility with the Hopf algebra structure. First in $\CC[H]$ we get the formulas :
$$\Delta(x_0^+) = 1\otimes x_0^+ + x_0^+\otimes k\text{ , }
\Delta(x_0^-) = x_0^-\otimes 1 + k^{-1}\otimes x_0^-,$$
$$\Delta(x_1^- k^{-1}) = (x_1^-k^{-1})\otimes k^{-1} + 1\otimes (x_1^- k^{-1})
\text{ , }\Delta(x_{-1}^+ k) = k\otimes (x_{-1}^+k) + (x_{-1}^+k)\otimes 1,$$
$$S(x_0^+) = - k^{-1} x_0^+\text{ , }S(x_0^-) = - k x_0^-\text{ , }S(x_1^- k^{-1}) = -x_1^-\text{ , }S(x_{-1}^+ k) = - x_{-1}^+.$$
But in $Z(\hat{\mathfrak{sl}_2})$, $x_0^+$, $x_0^-$, $x_1^-k^{-1}$, $x_{-1}^+k$ correspond to the Chevalley generators. As we have an isomorphism of Poisson algebras, they are also generators in $\CC[H]$ and we can conclude with Lemma \ref{isom}.
\qed

Remark: if we replace $SL_2$ by $GL_2$, we get an isomorphism with $\U_\epsilon(\hat{\mathfrak{gl}_2})$ instead of $\U_\epsilon(\hat{\mathfrak{sl}_2})$ where $\U_\epsilon(\hat{\mathfrak{gl}_2})$ is defined for example as in \cite[Section 3]{fr}. In this case additional coordinates $k_2$, $h_{2,r}$ are obtained from the coefficients of $\text{exp}(h_\pm^0)$.

\subsection{Symplectic leaves and elliptic curves} In this Section we study the symplectic leaves of $H$ and we see how they are related to $G$-bundle on $\mathcal{E}$.

Consider the map
$$\gamma : H\rightarrow \hat{SL_2}\text{ , }\gamma(a,b) = a^{-1}b.$$
The image of $\gamma$ consists of the elements $f\in \hat{SL_2}$ satisfying $n_1(f) = n_2(f) = 0$ and such that $f_-(\infty)f_+(0)$ lies in the big cell (here $n_1$, $n_2$, $f_+$, $f_-$ have been defined in Theorem \ref{decompb}). $\gamma$ is a $2$ to $1$ covering of its image. Indeed for $f^\pm\in\hat{SL_2}^{\pm}$ without constant term, $A\in SL_2$ in the big cell, $\Lambda,\Gamma\in\CC^*$, we have
$$\gamma^{-1}((f_+^{-1} A f_-,\Lambda^2,\Gamma^2))$$
$$= \{((\epsilon_1 B^{-1}f_+,\epsilon_2\Lambda^{-1},\epsilon_3\Gamma^{-1}),(\epsilon_1C f_-,\epsilon_2\Lambda,\epsilon_3\Gamma))|\epsilon_1,\epsilon_2,\epsilon_3\in(\ZZ/2\ZZ)^3\},$$
where
$$B = \begin{pmatrix} \beta & 0\\ A_{2,1}\beta^{-1}& \beta^{-1}\end{pmatrix}\text{ , }C = \begin{pmatrix} \beta & A_{1,2}\beta^{-1} \\ 0 & \beta^{-1} \end{pmatrix},$$
and $\beta, -\beta$ are the square roots of $A_{1,1}$.

We have the following situation :
$$\hat{SL_2}\overset{\gamma}{\leftarrow} H \rightarrow \tilde{SL_2}\rightarrow K\setminus \tilde{SL_2},$$
and let us denote $p:H\rightarrow K\setminus \tilde{SL_2}$ the restriction to $H$ of the projection to $K\setminus \tilde{SL_2}$. From Proposition \ref{symp} we have :

\begin{prop} The symplectic leaves of $H$ are the connected components of the preimages under $\gamma$ of conjugacy classes in $\hat{SL_2}$. \end{prop}

The value of $\Gamma$ is preserved along the symplectic leaves of $H$. As for $GL_1$, let us consider the generic case : $|\Gamma|\neq 1$. We suppose that $\Gamma$ is fixed (note that for the case $|\Gamma| = 1$, which is more complicated and not treated in the present paper, the classification of $q$-difference equations which is closely related to the classification of conjugacy orbits is discussed in \cite{dv}).

As mentioned before, several subgroups of $H$ are relevant from the representation theoretical point of view. Notably the subgroup $H_{hol}$ (resp. $H_{rat}$, $H_{mer}$, $H_{pos}$) of holomorphic maps on $\CC^*$ (resp. of  rational maps, of meromorphic maps, of germs well defined at $0$). We have corresponding symplectic leaves and analog subgroups of $\hat{SL_2}$ : $\hat{SL_2}_{hol}$, $\hat{SL_2}_{rat}$, $\hat{SL_2}_{mer}$, $\hat{SL_2}_{pos}$.

As an illustration let us concentrate on $\hat{SL_2}_{hol}$ (the other cases are analog, but additional analytic results should be proved to treat them with the same precision; for example more general analytic situations are considered in \cite{dv}).

\begin{thm}\label{sl} The holomorphic symplectic leaves are parameterized by :
$$\{\text{Isomorphism classes of holomorphic $SL_2$-bundles on }\mathcal{E}\}\times (\ZZ/2\ZZ)\times \CC^*.$$
\end{thm}

\demo The term $(\ZZ/2\ZZ)$ comes from the covering for the upper left coefficient of the matrix, and the term $\CC^*$ for the coefficient corresponding to $\Lambda$. Then the intersection of a conjugacy class in $\hat{SL_2}$ with $\hat{SL_2}_{hol}$ is equal to the conjugacy class in $\hat{SL_2}_{hol}$. Indeed from a relation $f(z\Theta) = h(z) f(z)g(z)^{-1}$ where $g, h$ are holomorphic on $\CC^*$, it is clear that $f$ is holomorphic on $\CC^*$ as $|\Gamma| \neq 1$. So the result follows from Theorem \ref{hlg}.
\qed

We do not describe here explicitly the symplectic leaves as for the $GL_1$ case, but general results for generic holomorphic leaves will be given in Section \ref{gent}. Let us give a few more precise results which can be proved in this case.

An element of $\hat{SL_2}$ is said to be lower triangular (resp. upper triangular, diagonal) if its first term has values in lower triangular (resp. upper triangular, diagonal) matrices. An element of $\hat{SL_2}$ is said to be constant if its first term is a constant germ. A symplectic leaf is said to be diagonal if its image by $\gamma$ is contained in a conjugacy class which contains a diagonal element (we use the same terminology for conjugacy classes). In the following, for $\alpha$ a germ of holomorphic map from the punctured disc to $\CC^*$ we denote $D_\alpha = \text{diag}(\alpha,\alpha^{-1})$. Consider $\mathcal{E}' = \mathcal{E}/(z\sim z^{-1})$.

\begin{lem}\label{param} The diagonal symplectic leaves are parameterized by $\mathcal{E}'\times (\ZZ/2\ZZ)$. A generic conjugacy class in $\hat{SL_2}$ containing a lower (resp. upper) triangular $f$ such that $n_1(f) = n_2(f) = 0$ is diagonal.
\end{lem}

\demo We prove that for $\alpha,\alpha'\in\CC^*$, $\Lambda,\Lambda'\in\CC^*$, the elements $(D_{\alpha},\Lambda,\Theta)$ and $(D_{\alpha'},\Lambda',\Theta)$ are in the same conjugacy class if and only if $\lambda = \lambda'$ and $\alpha\in (\alpha' \Theta^{\ZZ})\cup ((\alpha')^{-1}\Theta^{\ZZ})$.

First we look at the if part. From conjugation by $\begin{pmatrix}0& 1\\-1 & 0\end{pmatrix}$, it suffices to consider the case $\alpha\in \alpha' \Theta^{\ZZ}$. This case follows from conjugation by $D_z$ (the central extension part is not modified for such elements).

Now suppose that $(D_{\alpha},\Lambda,\Theta)$
and $(D_{\alpha'},\Lambda',\Theta)$ are conjugated by the element $(A,1,1)$ (we can suppose that the last two term are equal to $1$).

Suppose that $A_{2,1} = 0$. Then we have $\alpha A_{1,1}(z) A_{1,1}^{-1}(z\Theta) = \alpha'$. So $A_{1,1}(z) = \gamma z^n$ where $\gamma\in\CC^*$, $n\in\ZZ$ and we have $\alpha = \Theta^{n} \alpha'$. Let us look at the central extension term. We have $-\alpha \gamma z^n A_{1,2}(z\Theta) + \alpha^{-1} A_{1,2}(z) \gamma z^n \Theta^{n} = 0$. If $A_{1,2}(z) = 0$ the result follows from the discussion on the "if" part. If $A_{1,2}(z)\neq 0$, then $A_{1,2}(z) = \delta z^m$ where $\alpha^2 = \Theta^{(n - m)}$. But by the discussion for the "if" part, we can suppose that $\alpha^2 = 1$. So $n = m$ and $A = D_{z^n}\begin{pmatrix}\gamma & \delta\\ 0 & \gamma^{-1}\end{pmatrix}$. So the involved multiplications do not change the central charge and the result follows. The case $A_{1,2} = 0$ is treated in the same way.

Suppose that $A_{2,2} = 0$. Then we have $\alpha^{-1} A_{1,2}(z) A_{1,2}^{-1}(z\Theta) = \alpha'$. So $A_{1,2}(z) = \gamma z^n$ where $\gamma\in\CC^*$, $n\in\ZZ$ and we have $\alpha^{-1} = \Theta^{n} \alpha'$. Let us look at the central extension term. We have $-\alpha \gamma z^n A_{1,1}(z)\Theta^{n} + \alpha^{-1} A_{1,1}(z\Theta) \gamma z^n = 0$. If $A_{1,1} = 0$ we can conclude as above. If $A_{1,1}(z)\neq 0$, then $A_{1,1}(z) = \delta z^m$ where $\alpha^2 = \Theta^{(m - n)}$. As above, we can suppose that $\alpha^2 = 1$. So $n = m$ and $A = D_{z^n}\begin{pmatrix}\delta & \gamma\\ -\gamma^{-1} & 0\end{pmatrix}$. So the involved multiplications do not change the central charge and the result follows. The case $A_{1,1} = 0$ is treated in the same way.

Suppose that all coefficients of $A$ are non zero. From the relation
$$\alpha^2 A_{2,1}(z)A_{2,2}(\Theta) = A_{2,1}(z\Theta)A_{2,2}(z)$$
we get the existence of $n\in\ZZ$, $\gamma\in\CC^*$ such that $\alpha^2 = \Theta^{n}$ and $\gamma A_{2,1}(z) = A_{2,2}(z)z^n$. Then as above we can suppose that $\alpha^2 = 1$ and so $n = 0$. As above $\gamma' A_{1,2}(z) = A_{1,1}(z)$. Moreover $1 = A_{1,2}(z)A_{2,1}(z)(\gamma\gamma' - 1)$. Then we have $\alpha ' = \alpha\gamma\gamma'A_{1,2}(z) A_{2,1}(z\Theta)  - \alpha^{-1}A_{1,2}(z) A_{2,1}(z\Theta)$. As $\alpha = \alpha^{-1}$, we get $\alpha' = \alpha\frac{A_{1,2}(z)}{A_{1,2}(z\Theta)}$. So $\alpha' = \Theta^{M}$ and $A_{1,2}(z) = \delta z^M$. For the central extension term, we notice that
$$A = D_{z^M}\begin{pmatrix} \gamma'\delta& \delta\\ \delta^{-1}/(\gamma\gamma' - 1)&\gamma \delta^{-1}/(\gamma\gamma' - 1) \end{pmatrix},$$
and so we can conclude as above.

Let us prove the second statement. In a generic situation we can suppose that the constant part of the diagonal coefficients are not powers of $\Gamma^2$. We suppose that we have a lower triangular element $(A,\lambda,\Gamma^2)$ (the other case is analog). By conjugating by diagonal elements, we can see as in the proof of Proposition \ref{sympl} that we can suppose that $A_{1,1}$ is constant and not power of $\Gamma^2$. Then by conjugating by lower triangular elements with $1$ on the diagonal, as in the proof of Proposition \ref{sympl} we can suppose that $A_{2,1}$ is constant. The term $(A_{1,1} \Gamma^{2n}-A_{1,1}^{-1}\Gamma^{-2n})$ appears instead of $(\Gamma^{2n}-\Gamma^{-2n})$. As this term is never equal to $0$, we can choose $A_{2,1} = 0$.
\qed

So combined with Lemma \ref{param}, the generic triangular symplectic leaves are parameterized by $\mathcal{E}'\times (\ZZ/2\ZZ)$.

By this Lemma, to prove that an element $(A,\Lambda,\Theta)$ has a diagonal conjugacy class, it suffices to solve the following $q$-difference equation :
$$-A_{2,1} g(\Theta z)g(z) - A_{1,1}g(\Gamma^2 z) + A_{2,2} g(z) + A_{1,2} = 0.$$
(We refer to \cite{dvrsz, s} for general results on $q$-difference equations). Note that this equation is equivalent to an equation of the form $G(z\Theta)(1+G(z)) = \alpha(z)$.

\section{Geometric realization of the center}\label{gent}

Let us consider the case of a Lie group $G$ of arbitrary type. In addition to the coordinates for $\Lambda,\Gamma$ corresponding to the extensions, we define coordinates for the analytic loop group $L G$ corresponding to Drinfeld generators. This is totally analog to the case $SL_2$ studied above and the formula can be uniformly written in terms of root vectors of the affine Lie algebra. As the formula for these coordinates are explicitly written in \cite[(5.2.2), (5.2.4)]{bk}, we refer to them. We work with $\CC[H]$ the ring of maps which are polynomial in these coordinates. We have a structure of Poisson algebra on $\CC[H]$. It gives a geometric realization of $Z_\epsilon$ :

\begin{thm}\label{isomg} $\CC[H]$ is isomorphic to $Z_\epsilon$ as a Hopf Poisson algebra.
\end{thm}

\demo
Although as explained above the point of view used in \cite{bk} is different from the point of view developed in the present paper, the rank $2$ reduction argument of \cite{bk} makes perfectly sense in our situation. In fact the argument used in \cite[Section 5.3]{bk} is the following : the main point is that for $\alpha \in \Delta_+ + \NN\delta$ and $\alpha_i$ a simple root ($i\neq 0$), either $\alpha\in\alpha_i + \NN \delta$ or $(\alpha,\alpha_i)$ generates a subroot system of finite type. Then we have a reduction to finite rank $2$ case (treated in \cite{dckp}) and to the affine $\mathfrak{sl}_2$ and $\mathfrak{gl}_1$ cases which we have studied respectively in Theorem \ref{isomt} and Theorem \ref{isom1}.
\qed

As a consequence of Looijenga Theorem, we have as above :

\begin{thm}\label{gsl}
The holomorphic symplectic leaves are parameterized by :
$$\{\text{Isomorphism classes of holomorphic $G$-bundles on }\mathcal{E}\}\times (\mathbb Z/2\mathbb Z)^n\times\CC^*.$$
\end{thm}

Remark : the term $(\mathbb Z/2\mathbb Z)^n$ comes from the fact that $\gamma$ defines a $2^n$ to $1$ covering, as for the finite type in \cite{dcp} (this is deduced from the finite type as a germ of holomorphic map on the punctured disc which is in both part of the Riemann-Hilbert factorization is holomorphic on $\mathbb{P}_1(\CC)$ and so is constant).

In the case of $GL_1$ we have computed explicitly the generic symplectic leaves. In general we have for $G$ connected the following :

\begin{thm}\label{semis} A generic holomorphic symplectic leaf contains a constant element with value in $D$.
\end{thm}

\demo As a consequence of general results in \cite{ns, ra} about vector bundles on elliptic curves, it is proved in \cite[Cor. 3.5, Prop. 3.6]{efk} in the holomorphic case that almost all elements are twisted conjugated to an constant element with value in $D$.
\qed

\section{Conclusion}\label{conc}

We can now come back to the general picture described in the introduction, and we have a proof for the correspondence :

\begin{equation*}
\begin{split}
&\text{ Isomorphism classes of $G$-bundle on an elliptic curve,}
\\ \leftrightarrow &\text{ (by Looijenga Theorem) Equivalence classes of $q$-difference equations,}
\\ \leftrightarrow &\text{ (by trivial change of variable) Twisted conjugation classes in loop groups,}
\\ \leftrightarrow &\text{ (by the double construction) Symplectic leaves in loop groups,}
\\ \leftrightarrow &\text{ (by isomorphism) Equivalence classes of central characters of $\U_\epsilon(\hat{\Glie})$.}
\end{split}
\end{equation*}

Although in the present paper we focused on the structure and the geometry of the relevant loop groups itself, as explained in the introduction we have in mind applications to the representation theory of quantum affine algebras at roots of unity which will be discussed in a separate publication.

Central characters of $\U_\epsilon(\hat{\Glie})$ classify irreducible representations up to a finite cover. A very rich representation theory occurs in this context : in general infinite dimensional representations do appear, as for example baby Verma modules and Wakimoto modules. This will be explained in more details in another paper, but as for the finite type case explained in Section \ref{finite}, we can say that the symplectic leaves correspond to the action of a group of automorphisms of the quantum affine algebra, and so the statement of Theorem \ref{sl} is interpreted as a parametrization of equivalence classes of certain representations by $G$-bundles on the elliptic curve $\mathcal{E}$.

Moreover let us consider the category of representations with a nilpotent action of the generators with negative degree ($E_0$ has a nilpotent action). Although in general the simple representations of this category are not finite dimensional, we have a well-defined notion of natural graded character in this category. As in the finite dimensional case we expect that all irreducible representations for generic central characters will have the same graded dimensions.  Note that the holomorphic symplectic leaves studied above will give representations in this category. For such a representation we have a top irreducible component for $\U_\epsilon(\Glie)\subset\U_\epsilon(\hat{\Glie})$ and so such representations can be parameterized by these $\U_\epsilon(\Glie)$-submodules. We plan to study the corresponding graded branching rules to $\U_\epsilon(\Glie)$ which make sense in this category. The baby Verma modules correspond to central characters which are diagonal (they correspond to loop with value in $D$). Theorem \ref{semis} indicates that baby Verma modules will give informations on simple representations with generic central character.

Our future program includes also to extend the results in \cite{dcprr} for $\U_\epsilon(\Glie)$ to analyse the graded decomposition numbers for the tensor products of these representations. In \cite{dcprr} one of the main point is that $\U_\epsilon(\Glie)$ is finite over its center; the quantum affine algebra $\U_\epsilon(\hat{\Glie})$ is not finite over its center but has a corresponding grading, and so this is a motivating example to study in this spirit a theory of algebras graded over their center (graded Azumaya algebras).


\begin{thebibliography}{99}

\bibitem[A]{a} {\bf M. Atiyah}, {\it Vector bundles over an elliptic curve}, {Proc. London Math. Soc. (3) {\bf 7} 414--452 (1957)}

\mk

\bibitem[ADCKP]{adc} {\bf E. Arbarello, C. De Concini, V. Kac and C. Procesi}, {\it Moduli spaces of curves and representation theory}, {Comm. Math. Phys. {\bf 117}, no. 1, 1--36 (1988)}

\mk

\bibitem[AJS]{ajs} {\bf H. Andersen, J. Jantzen and W. Soergel}, {\it Representations of quantum groups at a $p$th root of unity and of semisimple groups in characteristic $p$: independence of $p$}, {Ast\'erisque  No. {\bf 220}, 321 pp. (1994)}

\mk

\bibitem[Be1]{b} {\bf J. Beck}, {\it Convex bases of PBW type for quantum affine algebras}, {Comm. Math. Phys. {\bf 165},  no. 1, 193--199 (1994)}

\mk

\bibitem[Be2]{b2} {\bf J. Beck}, {\it Representations of quantum groups at even roots of unity}, {J. Algebra {\bf 167} (1994), no. 1, 29--56}

\mk

\bibitem[Bi1]{bi1} {\bf G. Birkhoff}, {\it Singular points of ordinary linear differential equations}, {Trans. Amer. Math. Soc. {\bf 10}, no. 4, 436--470 (1909)}

\mk

\bibitem[Bi2]{bi2} {\bf G. Birkhoff}, {\it Equivalent singular points of ordinary linear differential equations}, {Math. Ann. {\bf 74},  no. 1, 134--139 (1913)}

\mk

\bibitem[Bo]{bo} {\bf P. Boalch}, {\it G-bundles, isomonodromy, and quantum Weyl groups}, {Int. Math. Res. Not. {\bf 2002}, no. 22, 1129--1166}

\mk

\bibitem[BEG]{beg} {\bf V. Baranovsky, S. Evens and V. Ginzburg}, {\it Representations of quantum tori and $G$-bundles on elliptic curves}, {The orbit method in geometry and physics (Marseille, 2000), 29--48,
Progr. Math., 213, Birkhauser Boston, Boston, MA, 2003}

\mk

\bibitem[BG]{bg} {\bf V. Baranovsky and V. Ginzburg}, {\it Conjugacy classes in loop groups and $G$-bundles on elliptic curves}, {Internat. Math. Res. Notices {\bf 1996},  no. 15, 733--751}

\mk

\bibitem[BaK]{bak} {\bf E. Backelin and K. Kremnizer}, {\it Localization for quantum groups at a root of unity}, { J. Amer. Math. Soc. {\bf 21} (2008),  no. 4, 1001--1018.}

\mk

\bibitem[BeK]{bk} {\bf J. Beck and V. G. Kac}, {\it Finite-dimensional representations of quantum affine algebras at roots of unity}, {J. Amer. Math. Soc.  {\bf 9},  no. 2, 391--423 (1996)}

\mk

\bibitem[BMR]{bmr} {\bf R. Bezrukavnikov, I. Mirkovic and D. Rumynin}, {\it Localization of modules for a semisimple Lie algebra in prime characteristic}, {Ann. Math. {\bf 167}, no. 3 (2008)}

\mk

\bibitem[CJ]{cj} {\bf V. Chari and N. Jing,} {\it Realization of level one representations of $U_q(\hat{\Glie})$ at a root of unity} {Duke Math. J. {\bf 108} (2001), no. 1, 183--197}

\mk

\bibitem[CM]{cm} {\bf A. Connes and M. Marcolli}, {\it Noncommutative geometry, quantum fields and motives}, {American Mathematical Society Colloquium Publications, 55. AMS, Providence, RI; Hindustan Book Agency, New Delhi, 2008}

\mk

\bibitem[CP1]{cp1}{\bf V. Chari and A. Pressley}, {\it A Guide to Quantum Groups}, {Cambridge University Press, Cambridge (1994)}

\mk

\bibitem[CP2]{cp} {\bf V. Chari and A. Pressley}, {\it Quantum affine algebras at roots of unity}, {Represent. Theory {\bf 1}, 280--328 (1997)}

\mk

\bibitem[DC]{dc} {\bf C. De Concini}, {\it Poisson algebraic groups and representations of quantum groups at roots of $1$}, {in First European Congress of Mathematics, Vol. I (Paris, 1992),  93--119, Progr. Math. {\bf 119}, Birkhauser, Basel (1994)}

\mk

\bibitem[DCK]{dck} {\bf C. De Concini and V. Kac}, {\it Representations of quantum groups at roots of $1$}, {Operator algebras, unitary representations, enveloping algebras, and invariant theory (Paris, 1989),  471--506, Progr. Math., 92, Birkhauser Boston, Boston, MA, (1990)}

\mk

\bibitem[DCP]{dcp} {\bf C. De Concini and C. Procesi}, {\it Quantum groups}, {in $D$-modules, representation theory, and quantum groups (Venice, 1992),  31--140, Lecture Notes in Math. {\bf 1565}, Springer, Berlin (1993)}

\mk

\bibitem[DCKP]{dckpj} {\bf C. De Concini, V. Kac and C. Procesi}, {\it Quantum coadjoint action}, {J. Amer. Math. Soc. {\bf 5} (1992),  no. 1, 151--189}

\mk

\bibitem[DCKP2]{dckp} {\bf C. De Concini, V. Kac and C. Procesi}, {\it Some quantum analogues of solvable Lie groups}, {Geometry and analysis (Bombay, 1992),  41--65, Tata Inst. Fund. Res., Bombay (1995)}

\mk

\bibitem[DCPRR]{dcprr} {\bf C. De Concini, C. Procesi, N. Reshetikhin and M. Rosso},
{\it Hopf algebras with trace and representations}, {Invent. Math. {\bf 161}, no. 1, 1--44 (2005)}

\mk

\bibitem[DV]{dv} {\bf L. Di Vizio}, {\it Local analytic classification of q-difference equations with
$|q| = 1$}, {J. Noncommut. Geom. {\bf 3} (2009), no. 1, 125--149}

\mk

\bibitem[DVRSZ]{dvrsz} {\bf L. Di Vizio, J.-P. Ramis, J. Sauloy, C. Zhang}, {\it Equations aux $q$-diff\'erences} {Gaz. Math. No. {\bf 96}, 20--49 (2003)}

\mk

\bibitem[Da1]{da0} {\bf I. Damiani}, {\it A basis of type Poincar\'e-Birkhoff-Witt for the quantum algebra of $\widehat{\rm sl}(2)$}, {J. Algebra {\bf 161}  (1993),  no. 2, 291--310}

\mk

\bibitem[Da2]{da} {\bf I. Damiani}, {\it The highest coefficient of ${\rm det}\,H\sb \eta$ and the center of the specialization at odd roots of unity for untwisted affine quantum algebras}, {J. Algebra {\bf 186},  no. 3, 736--780 (1996)}

\mk

\bibitem[Dr2]{Dri2}{\bf V. G. Drinfeld}, {\it A new realization of Yangians and of quantum affine algebras}, {Soviet Math. Dokl. {\bf 36}, no. 2, 212--216 (1988)}

\mk

\bibitem[En1]{e1} {\bf B. Enriquez}, {\it Integrity, integral closedness and finiteness over their centers of the coordinate algebras of quantum groups at $p\sp \nu$th roots of unity}, {Ann. Sci. Math. Qu\'ebec {\bf 19},  no. 1, 21--47 (1995)}

\mk

\bibitem[En2]{e2} {\bf B. Enriquez}, {\it Le centre des alg\`ebres de coordonn\'ees des groupes quantiques aux racines $p\sp \alpha$-i\`emes de l'unit\'e}, {Bull. Soc. Math. France {\bf 122},  no. 4, 443--485 (1994)}

\mk

\bibitem[Et]{et} {\bf P. Etingof}, {\it Quantum Knizhnik-Zamolodchikov equations and holomorphic vector bundles}, {Duke Math. J. {\bf 70}, no. 3, 591--615 (1993)}

\mk

\bibitem[EF]{ef}{\bf P. Etingof and I. Frenkel}, {\it Central extensions of current groups in two dimensions}, {Comm. Math. Phys. {\bf 165} (1994),  no. 3, 429--444}

\mk

\bibitem[EFK]{efk}{\bf P. Etingof, I. Frenkel and A. Kirillov}, {\it Spherical functions on affine Lie groups}, {Duke Math. J. {\bf 80}  (1995),  no. 1, 59--90}

\mk

\bibitem[ES]{es}{\bf P. Etingof and O. Schiffmann}, {\it Lectures on quantum groups. Second edition}, {Lectures in Mathematical Physics. International Press, Somerville, MA, 2002}

\mk

\bibitem[Fa1]{fa} {\bf G. Faltings}, {\it A proof for the Verlinde formula}, {J. Algebraic Geom. {\bf 3}, no. 2, 347--374 (1994)}

\mk

\bibitem[Fa2]{fad} {\bf G. Faltings}, {\it Algebraic loop groups and moduli spaces of bundles}, {J. Eur. Math. Soc. {\bf 5},  no. 1, 41--68 (2003)}

\mk

\bibitem[F]{f2} {\bf E. Frenkel}, {\it Langlands correspondence for loop groups}, {Cambridge Studies in Advanced Mathematics, 103. Cambridge University Press, Cambridge (2007)}

\mk

\bibitem[FF]{ff} {\bf B. Feigin and E. Frenkel}, {\it Affine Kac-Moody algebras and semi-infinite flag manifolds}, {Comm. Math. Phys. {\bf 128}, no. 1, 161--189 (1990)}

\mk

\bibitem[FM]{fm} {\bf E. Frenkel and E. Mukhin}, {\it The $q$-characters at roots of unity}, {Adv. Math. {\bf 171},  no. 1, 139--167 (2002)}

\mk

\bibitem[FMW]{fmw} {\bf R. Friedman, J. Morgan and E. Witten} {\it Principal $G$-bundles over elliptic curves}, {Math. Res. Lett. {\bf 5},  no. 1-2, 97--118 (1998)}

\mk

\bibitem[FR]{fr} {\bf E. Frenkel and N. Reshetikhin}, {\it Quantum affine algebras and deformations of the Virasoro and $W$-algebras} {Comm. Math. Phys. {\bf 178},  no. 1, 237--264 (1996)}

\mk

\bibitem[H]{h} {\bf D. Hernandez}, {\it The $t$-analogs of $q$-characters at roots of unity for quantum affine algebras and beyond}, {J. Algebra {\bf 279},  no. 2, 514--557 (2004)}

\mk

\bibitem[HSW]{hsw} {\bf N. Hitchin, G. Segal and R. Ward}, {\it Integrable systems.
Twistors, loop groups, and Riemann surfaces} {Lectures from the Instructional Conference held at the University of Oxford, Oxford, September 1997. Oxford Graduate Texts in Mathematics, 4. The Clarendon Press, Oxford University Press, New York (1999)}

\mk

\bibitem[Ka]{kac} {\bf V. Kac}, {\it Infinite dimensional Lie algebras}, {3rd Edition, Cambridge University Press (1990)}

\mk

\bibitem[Ku]{ku} {\bf S. Kumar}, {\it Kac-Moody groups, their flag varieties and representation theory}, {Progress in Mathematics, 204. Birkhauser Boston, Inc., Boston, MA (2002)}

\mk

\bibitem[KP]{kp} {\bf V. Kac and D. Peterson}, {\it Defining relations of certain infinite-dimensional groups}, {The mathematical heritage of Elie Cartan (Lyon, 1984). Ast\'erisque 1985,  Numero Hors Serie, 165--208}

\mk

\bibitem[La]{l} {\bf Y. Laszlo}, {\it About $G$-bundles over elliptic curves}, {Ann. Inst. Fourier (Grenoble) {\bf 48}, no. 2, 413--424 (1998)}

\mk

\bibitem[Lu1]{lj} {\bf G. Lusztig}, {\it Finite-dimensional Hopf algebras arising from quantized universal enveloping algebra}, {J. Amer. Math. Soc. {\bf 3}  (1990),  no. 1, 257--296}

\mk

\bibitem[Lu2]{l2} {\bf G. Lusztig}, {\it Introduction to quantum groups}, {Progress in Mathematics, 110. Birkhauser Boston, Inc., Boston, MA (1993)}

\mk

\bibitem[N]{n2} {\bf H. Nakajima}, {\it Quiver varieties and $t$-analogs of $q$-characters of quantum affine algebras}, {Ann. of Math. (2)  {\bf 160},  no. 3, 1057--1097 (2004)}

\mk

\bibitem[NS]{ns} {\bf M. S. Narasimhan and C. S. Seshadri}, {\it Stable and unitary vector bundles on a compact Riemann surface}, {Ann. of Math. (2) {\bf 82} 540--567 (1965)}

\mk

\bibitem[PS]{ps} {\bf A. Pressley and G. Segal}, {\it Loop groups}, {Oxford Mathematical Monographs. Oxford Science Publications. The Clarendon Press, Oxford University Press, New York (1986)}

\mk

\bibitem[Ra]{ra} {\bf A. Ramanathan}, {\it Stable principal bundles on a compact Riemann surface}, {Math. Ann. {\bf 213}, 129--152 (1975)}

\mk

\bibitem[Re]{r} {\bf N. Reshetikhin}, {\it Quasitriangularity of quantum groups at roots of $1$}, {Comm. Math. Phys. {\bf 170} (1995), no. 1, 79--99}

\mk

\bibitem[S]{s} {\bf J. Sauloy}, {\it Equations aux q-differences et fibres vectoriels holomorphes sur la courbe elliptique $\CC^*/q^\ZZ$}, {Ast\'erisque No. {\bf 323} (2009), 397--429.}

\end{thebibliography}
\end{document}